\documentclass[10pt,a4paper,american]{amsart}
\usepackage{babel}
\usepackage{amsmath,amsthm,bm}
\usepackage{amsbsy}
\usepackage{amstext}
\usepackage{amsfonts}
\usepackage{floatflt}
\usepackage[pdftex]{graphicx}
\usepackage{mathcomp}
\usepackage{mathtools}
\usepackage[dvipsnames]{xcolor}
\usepackage{dsfont}
\usepackage{latexsym}
\usepackage{amssymb}
\usepackage{stmaryrd}
\usepackage{bm}
\usepackage{enumerate}
\usepackage{esint}	
\usepackage[colorlinks,pdfpagelabels,pdfstartview = FitH,bookmarksopen = true,bookmarksnumbered = true,linkcolor = blue,plainpages = false,hypertexnames = false,citecolor = red,pagebackref=false]{hyperref}
\usepackage{cite}

\allowdisplaybreaks
\makeatletter
\newsavebox{\@brx}
\newcommand{\llangle}[1][]{\savebox{\@brx}{\(\m@th{#1\langle}\)}%
  \mathopen{\copy\@brx\kern-0.5\wd\@brx\usebox{\@brx}}}
\newcommand{\rrangle}[1][]{\savebox{\@brx}{\(\m@th{#1\rangle}\)}%
  \mathclose{\copy\@brx\kern-0.5\wd\@brx\usebox{\@brx}}}
\makeatother

\newtheorem{theorem}{Theorem}
\newtheorem{lemma}[theorem]{Lemma}
\newtheorem{proposition}[theorem]{Proposition}

\theoremstyle{definition}
\newtheorem{definition}[theorem]{Definition}

\theoremstyle{remark}
\newtheorem{remark}[theorem]{Remark}

\numberwithin{theorem}{section}
\numberwithin{equation}{section}

\DeclareMathOperator{\Div}{div}
\DeclareMathOperator{\dist}{dist}

\newcommand{\N}{\ensuremath{\mathbb{N}}}
\newcommand{\R}{\ensuremath{\mathbb{R}}}
\newcommand{\mint}{- \mskip-19,5mu \int}

\newcommand{\ca}{\operatorname{cap}}

\newcommand{\dx}{\mathrm{d}x}
\newcommand{\dt}{\mathrm{d}t}

\makeatletter
\@namedef{subjclassname@2020}{%
  \textup{2020} Mathematics Subject Classification}
\makeatother

\newcommand{\babs}[1]{\big|#1\big|}

\def\Xint#1{\mathchoice
    {\XXint\displaystyle\textstyle{#1}}%
    {\XXint\textstyle\scriptstyle{#1}}%
    {\XXint\scriptstyle\scriptscriptstyle{#1}}%
    {\XXint\scriptscriptstyle\scriptscriptstyle{#1}}%
    \!\int}
\def\XXint#1#2#3{\setbox0=\hbox{$#1{#2#3}{\int}$}
    \vcenter{\hbox{$#2#3$}}\kern-0.5\wd0}
\def\bint{\Xint-}
\def\dashint{\Xint{\raise4pt\hbox to7pt{\hrulefill}}}

\def\Xiint#1{\mathchoice
    {\XXiint\displaystyle\textstyle{#1}}%
    {\XXiint\textstyle\scriptstyle{#1}}%
    {\XXiint\scriptstyle\scriptscriptstyle{#1}}%
    {\XXiint\scriptscriptstyle\scriptscriptstyle{#1}}%
    \!\iint}
\def\XXiint#1#2#3{\setbox0=\hbox{$#1{#2#3}{\iint}$}
    \vcenter{\hbox{$#2#3$}}\kern-0.5\wd0}
\def\biint{\Xiint{-\!-}}

\renewcommand{\epsilon}{\varepsilon}
\newcommand{\eps}{\varepsilon}
\renewcommand{\rho}{\varrho}

\renewcommand{\epsilon}{\varepsilon}
\renewcommand{\rho}{\varrho}
\renewcommand{\d}{\:\! \mathrm{d}}

\newcommand{\A}{\mathbf{A}}

\numberwithin{equation}{section}
\allowdisplaybreaks


\begin{document}
\renewcommand{\refname}{References} 
\renewcommand{\abstractname}{Abstract} 
\title[Higher integrability in noncylindrical domains]{Global higher integrability for systems with $p$-growth structure in noncylindrical domains}

\author[K.~Moring]{Kristian Moring}
\address{Kristian Moring\\
Fachbereich Mathematik, Paris-Lodron Universit\"at Salzburg\\
Hellbrunner Str.~34, 5020 Salzburg, Austria}
\email{kristian.moring@plus.ac.at}

\author[L.~Schätzler]{Leah Schätzler}
\address{Leah Sch\"atzler\\
Department of Mathematics and Systems Analysis, Aalto University\\
P.O.~Box 11100, FI-00076 Aalto, Finland}
\email{ext-leah.schatzler@aalto.fi}

\author[C.~Scheven]{Christoph Scheven}
\address{Christoph Scheven\\
Fakultät für Mathematik, Universität Duisburg-Essen\\
Thea-Leymann-Str.~9, 45127 Essen, Germany}
\email{christoph.scheven@uni-due.de}

\date{}

\begin{abstract}
We consider the Cauchy--Dirichlet problem to systems with $p$-growth structure with $1 < p < \infty$, whose prototype is
$$
	\partial_t u- \Div \big( |Du|^{p-2} Du \big) = \Div \left( |F|^{p-2} F \right),
$$
in a bounded noncylindrical domain $E \subset \mathds{R}^{n+1}$.
For $p> \frac{2(n+1)}{n+2}$ and domains $E$ that satisfy suitable regularity assumptions and do not grow or shrink too fast, we prove global higher integrability of $Du$.
The result is already new in the case $p=2$.
\end{abstract}
\makeatother

\maketitle

\section{Introduction}
For $n \in \N$ and $T>0$, we suppose that $E \subset \R^n \times [0,T)$ is a bounded relatively open domain, and we consider the Cauchy--Dirichlet problem to partial differential equations (PDEs) of parabolic $p$-Laplace type
$$
	\partial_t u- \Div \big( |Du|^{p-2} Du \big) = \Div \left( |F|^{p-2} F \right)
	\qquad \text{in } E,
$$
where $1 < p < \infty$, $u \colon E \to \R^N$ for some $N \geq 1$, and $F \in L^p(E,\R^{Nn})$.
In particular, in the case $N>1$, we are dealing with systems of PDEs.

More generally, denoting the initial and boundary values by $g_o$ and $g$, respectively, we are concerned with systems with $p$-growth structure of the form
\begin{align} \label{eq:pde}
\left\{
\begin{array}{rl}
\partial_t u- \Div \mathbf{A}(x,t,u, D u) = \Div \left( |F|^{p-2} F \right) &\text{ in } E, \vspace{1mm}\\
u=g &\text{ on } \bigcup_{t \in [0,T)} \partial E^t \times \{t\}, \vspace{1mm}\\
u=g_o &\text{ on } E^0 \times \{0\}.
\end{array}
\right. 
\end{align}
For the precise assumptions on $E$, $\mathbf{A}$, $g_o$, and $g$, we refer to Section \ref{sec:setting}.
Moreover, the time slice of $E$ at a fixed time $t \in [0,T)$ is given by
$$
E^t = \{x \in \R^n: (x,t) \in E\},
$$
such that 
$$
E = \bigcup_{t\in [0,T)} E^t\times\{t\}.
$$
Viewing $E^t$ as a spatial set that changes over time instead of focusing on $E$ as a \emph{noncylindrical domain} in space-time, the terms \emph{time-varying} or \emph{time-changing domains} have also been used in the literature.
For the existence theory for PDEs of parabolic $p$-Laplace type, and some generalizations in such domains, we refer to \cite{BDSS,Calvo-Novaga-Orlandi,Lan-etal,Paronetto,SSSS-nondecreasing,SSSS-general}.

The motivation for considering PDEs in noncylindrical domains is twofold.
On the one hand, they appear naturally in some physical and biological applications, e.g.~the flow of fluids through a container whose walls can be moved, or pattern formation in a growing organism.
For an overview of applications of different PDEs involving a noncylindrical setting, we refer to~\cite{Crampin-Gaffney-Maini,Knobloch-Krechetnikov,Krechetnikov,Murray,Shelley-Tiany-Wlodarski,Stefan}.
On the other hand, from a purely mathematical point of view, the noncylindrical formulation truly takes the evolutionary nature of the parabolic system \eqref{eq:pde} into account, whereas the cylindrical case $E = \Omega \times [0,T)$ constitutes a simpler special case.

In the present paper, we establish the global higher integrability of the gradient of weak solutions $u$ to \eqref{eq:pde}.
This means that while a priori we only know that $|Du|$ is integrable to power $p$, under suitable conditions on $E$, $g_o$, $g$, and $F$ we prove that $|Du|$ is actually integrable to a higher exponent than $p$.
To the best of our knowledge, our result is already new for the case $p=2$ in the noncylindrical situation.

Before we state the precise result, we briefly discuss the history of the problem.
For elliptic PDEs of $p$-Laplace type in a domain $\Omega \subset \R^n$, local higher integrability is due to Meyers \& Elcrat \cite{Meyers-Elcrat}, while global higher integrability was established by Kilpeläinen \& Koskela \cite{Kilpelainen-Koskela}.
In particular, the authors of the latter article assume that the complement of $\Omega$ satisfies a \emph{uniform $p$-fatness condition} (see Definition \ref{def:unif-p-fat} below), and show that this condition is essentially optimal.
A key component in their proof is the self-improving property of the uniform $p$-fatness condition that was first established by Lewis \cite{Lewis} in the more general context of Riesz potentials (see Lemma \ref{lem:fatness-larger-exponents} below).

For PDEs of parabolic $p$-Laplace type with quadratic growth $p=2$, Giaquinta \& Struwe \cite{Giaquinta-Struwe} prove local higher integrability.
However, their technique does not carry over to the case of more general $p$.
The breakthrough was achieved almost 20 years later by Kinnunen \& Lewis \cite{Kinnunen-Lewis:1,Kinnunen-Lewis:very-weak}, who were able to deal with the whole range $\max\big\{ 1, \frac{2n}{n+2} \big\} < p < \infty$.
Note that the critical exponent $\frac{2n}{n+2}$ appears naturally in different areas of regularity theory for the parabolic $p$-Laplace equation.
Their key idea was to use a suitable \emph{intrinsic geometry} to balance the different scalings in the evolution and diffusion terms, i.e., to consider cylinders of the form $Q_{\varrho}^{(\lambda)}(x_o,t_o) := B_\varrho(x_o) \times (t_o -\lambda^{2-p} \varrho^2, t_o + \lambda^{2-p} \varrho^2)$, whose length depends on the integral average of $|Du|^p$ over the same cylinder via
$$
    \lambda^p \approx \biint_{Q_{\varrho}^{(\lambda)}(x_o,t_o)} |Du|^p \,\dx\dt.
$$
The main tool in the proof of higher integrability is a \emph{reverse Hölder inequality} for solutions in such intrinsic cylinders, meaning that the integral average of $|Du|^p$ over an intrinsic cylinder can be estimated from above in terms of the integral average of $|Du|^q$ with a smaller exponent $q<p$.
Subsequently, Parviainen \cite{Parviainen-degenerate,Parviainen-singular} established global higher integrability in the cylindrical setting $\Omega \times [0,T)$ under the condition that $\R^n \setminus \Omega$ is uniformly $p$-fat.
Since this condition is analogous to the elliptic case, it is optimal.

In the noncylindrical setting, we are able to deal with the same level of generality concerning the spatial regularity of $E$.
In addition, we give sufficient assumptions on the speed at which $E$ may grow or shrink in time.
It is an interesting question for future research to determine the optimal condition on $E$ with respect to the time variable that allows for global higher integrability.
To state our main result, we need the following notation.
For $z_o = (x_o,t_o) \in \R^{n+1}$ and $R > 0$, we denote standard parabolic cylinders by
$$
Q_R := Q_{R}(z_o) = B_R (x_o) \times \big(t_o - R^2 , t_o + R^2 \big).
$$
Further, throughout the paper, let
\begin{equation}
	G := |D g|^p + |\partial_t g|^{\hat p'} + |F|^p,
	\label{eq:definition-G}
\end{equation}
where $\hat p'$ is defined in~\eqref{def-hat-p}.
Moreover, letting $\beta$ given by \eqref{eq:beta-range} be the exponent in the two-sided condition \eqref{assumption:two-sided-growth}, the \emph{scaling deficit} is given by
\begin{equation}
	d=
	\left\{
	\begin{array}{ll}
		\frac{p}{2} & \text{for } p \geq 2, \\[5pt]
      	\frac{p (2 \beta-1)}{\beta( p(n+2) - 2n) - 2 } & \text{for } \frac{2(n+1)}{n+2} < p < 2.
	\end{array}
	\right.
   \label{eq:scaling-deficit}
\end{equation}
Then, our main theorem is the following.
\begin{theorem}\label{thm:main}
Let $\frac{2(n+1)}{n+2} < p <\infty$, suppose that $\R^n \setminus
E^t$ is uniformly $p$-fat in the sense of Definition
\ref{def:unif-p-fat} with a fatness constant $\alpha>0$ for every $t \in (0,T)$, and assume that $E$ satisfies the conditions \eqref{assumption:one-sided-growth} with $\ell$ given by \eqref{assumption:function-one-sided-growth}, and \eqref{assumption:two-sided-growth} with some $\beta$ in the range given by \eqref{eq:beta-range}.
Furthermore, let the operator $\mathbf{A}$ be given according to
\eqref{assumption:A}, and assume that for some $\sigma > 0$, the
lateral and initial boundary values $g \in \mathfrak{G}^{\sigma}$ and
$g_o \in \mathfrak{G}^{\sigma}_o$ \textup{(}see \eqref{def:G-sigma} and \eqref{def:G-sigma-o} for the definition of these spaces\textup{)} satisfy \eqref{compatibility-lateral-initial}, and the source term $F \in L^{p+\sigma}(E,\R^{Nn})$. 
Finally, let $u$ be a weak solution to~\eqref{eq:pde} according to Definition~\ref{def:weak-sol-global}.
Then, there exists $\eps_o = \eps_o(n,p,C_o,C_1,M, \alpha, \beta) >0$ such that 
$$
Du \in L^{p+\eps_1}(E, \R^{Nn}),
$$
where $\eps_1 = \min\{\eps_o, \sigma\}$. Furthermore, for every $\eps \in (0,\eps_o]$ and any cylinder $Q_{2R}(z_o) \subset \R^{n+1}$ with $z_o \in \overline{E}$, the bound
\begin{align*}
	\biint_{Q_{R}} |Du|^{p+\eps} \chi_E\, \d x \d t
	&\leq
	c \left(1+ \biint_{Q_{4R}} (|Du|^p + G  )\chi_E \, \d x \d t \right)^{1+\frac{\eps d}{p}}
	\\ &\quad
	+ c \biint_{Q_{2R}} G^{1+\frac{\eps}{p}} \chi_E \, \d x \d t
	\\ & \quad
	+ c  \left( \bint_{ B_{2R} } |D g_o|^{{p}_*+\eps} \chi_{E^0} \, \d x \right)^{\frac{\hat{p} +\eps }{\hat{p}_*+\eps}}
\end{align*}
holds true for some $c = c(n,p,C_o,C_1,M,\alpha, \beta) > 0$, where $G$ is defined according to \eqref{eq:definition-G} and $d$ is given by \eqref{eq:scaling-deficit}.
\end{theorem}

While we follow the same overall strategy as in the cylindrical setting, two key steps in the proof of Theorem \ref{thm:main} rely on new ideas.
First of all, the integration by parts formula in Proposition \ref{prop:ineq_integration_by_parts_formula}, that was first developed in \cite{BDSS} and then generalized in \cite{SSSS-general}, is a crucial component in the proof of the energy estimate near the lateral boundary of $E$ in Lemma \ref{lem:Caccioppoli-boundary}.
Note that it is applicable, since any weak solution in the sense of Definition \ref{def:weak-sol-global} possesses a distributional time derivative, see Remark \ref{rem:time-derivative}.
In turn, the integration by parts formula is based on the one-sided growth condition \eqref{assumption:one-sided-growth} with $\ell$ given by \eqref{assumption:function-one-sided-growth}.
Note that it would be interesting to determine whether the lower bound $p> \frac{2(n+1)}{n+2}$ appearing in this context is optimal.

The next new idea concerns the precise distinction between lateral, initial and interior cylinders, and the proof of reverse Hölder inequalitites for lateral cylinders in the noncylindrical setting, see Lemma \ref{lem:reverse-holder-lateral}.
In particular, in this context we need to ensure that all time slices of an intrinsic cylinder $Q_{\varrho}^{(\lambda)}(x_o,t_o)$ are close to the lateral boundary $\bigcup_{t \in [0,T)} \partial E^t \times \{t\}$ if this is the case for one of the time slices of the cylinder.
In this context, the two-sided growth condition \eqref{assumption:two-sided-growth} with $\beta$ in the range given by \eqref{eq:beta-range} comes into play.
Again, while the lower bound $p> \frac{2(n+1)}{n+2}$ is vital in our approach, it would be interesting to determine whether it is optimal.

The article is organized as follows.
In Section \ref{sec:preliminaries} we give the precise setting, clarify notation, and collect auxiliary results.
Next, in Section \ref{sec:Caccioppoli} we establish an energy estimate, which is used in Section \ref{sec:rev-Holder} to derive a reverse Hölder inequality close to the lateral boundary.
Moreover, we recall the corresponding results for cylinders away from the lateral boundary from the literature.
We conclude the proof of Theorem \ref{thm:main} by standard arguments in Section \ref{sec:proof-of-mainthm}.

\section{Preliminaries}
\label{sec:preliminaries}

\subsection{Setting}
\label{sec:setting}
In this section, we give the precise setting.
First of all, we assume that the complement of the time slices of $E$ is uniformly $p$-fat in the following sense.
\begin{definition} \label{def:unif-p-fat}
A set $A \subset \R^n$ is uniformly $p$-fat if there exists a constant $\alpha > 0$ such that
$$
\ca_p \big(A \cap \overline{B}_\rho(x), B_{2\rho}(x) \big) \geq \alpha \ca_p \big(\overline{B}_\rho(x), B_{2\rho}(x) \big) 
$$
for every $x \in A$ and $\rho > 0$.
\end{definition}

Moreover, we need to control the speed at which $E$ is allowed to grow and shrink over time.
To this end, we denote the complementary excess by
$$
\mathbf{e}^c (E^s,E^t) := \sup_{x \in \R^n \setminus E^t} \dist (x, \R^n \setminus E^s),\quad \text{ for } s,t \in [0,T).
$$
We ensure that $E$ does not grow too fast by assuming that
\begin{align}\label{assumption:one-sided-growth}
    \mathbf{e}^c \big(E^{t}, E^{s} \big) \le \vert \ell(t)-\ell(s) \vert
    \quad \text{for $0 \le s \le t <T$,}
\end{align}
where the function $\ell \colon (-1,T+1)\rightarrow(0,\infty)$ satisfies
\begin{align}\label{assumption:function-one-sided-growth}
  \ell\in W^{1,r}(-1,T+1)
  \mbox{\quad for\quad}
  \left\{
  \begin{array}{ll}
    r = \frac{p}{p-1}, & \text{if }p \geq 2,\\[1.5ex]
    r = \frac{p(n+2)-2n}{p(n+2)-2(n+1)}, & \text{if } \frac{2(n+1)}{n+2} < p <2.
  \end{array}
  \right.
\end{align}
In addition to~\eqref{assumption:one-sided-growth}, we impose the two-sided condition
\begin{equation} \label{assumption:two-sided-growth}
\mathbf{e}^c (E^s,E^t) \leq M |t-s|^\beta \quad \text{ for all } s,t \in [0, T)
\end{equation}
for some $M>0$ (without loss of generality, we assume that $M \geq 1$ throughout the paper), and $\beta$ satisfying
\begin{equation} \label{eq:beta-range}
\beta \in \left( \max \left\{\tfrac12, \tfrac{2}{p(n+2) - 2n} \right\},1\right].
\end{equation}
Note that the interval in \eqref{eq:beta-range} is nonempty when $p > \frac{2(n+1)}{n+2}$. 
\begin{remark}
At this point, we would like to discuss why we assume both the one-sided condition \eqref{assumption:one-sided-growth} and the two-sided condition \eqref{assumption:two-sided-growth} on $E$ instead of imposing only one condition on the speed at which $E$ can grow and shrink, respectively.
On the one hand, we cannot omit \eqref{assumption:one-sided-growth} with $\ell$ given by~\eqref{assumption:function-one-sided-growth}, since the Hölder type condition \eqref{assumption:two-sided-growth} does not imply the existence of the required time derivative $\ell'$.
On the other hand, \eqref{assumption:one-sided-growth} with $\ell$ given by~\eqref{assumption:function-one-sided-growth} implies a one-sided condition of the type~\eqref{assumption:two-sided-growth} with $0\leq t \leq s < T$, $M =  \|\ell'\|_{L^r(-1,T+1)}$, and $\beta = \frac{2}{p(n+2)-2n} > \tfrac12$ in the case $\frac{2(n+1)}{n+2} < p < 2$, and $\beta = \tfrac1p \leq \tfrac12$ in the case $p \geq 2$. However, in the proof of Theorem~\ref{thm:main} we need that $\beta > \tfrac12$ in all cases, see e.g. Lemma~\ref{lem:ball-complement-intersect}. Moreover, in the case $\tfrac{2}{p(n+2)-2n}<p<2$, we assume the strict inequality $\beta > \tfrac{2}{p(n+2)-2n}$ in order to guarantee that the scaling deficit $d$, defined in~\eqref{eq:scaling-deficit}, is positive and well defined in order to exploit a suitable stopping time argument in the beginning of Section~\ref{sec:proof-of-mainthm}.
\end{remark}

Next, we suppose that in~\eqref{eq:pde}, $\A \colon E\times\R^N\times\R^{Nn}\to\R^{Nn}$ is a Carath\'eodory
function satisfying
\begin{align}
\label{assumption:A}
\left\{
\begin{array}{c}
\A(x,t,u,\xi)\cdot \xi \geq C_o |\xi|^p, \\[5pt]
|\A(x,t,u, \xi)| \leq C_1|\xi|^{p-1}
\end{array}
\right.
\end{align}
for a.e.~$(x,t) \in E$ and all $(u,\xi) \in \R^N \times \R^{Nn}$ with constants $0<C_o \leq C_1$.

We will consider exponents
\begin{equation}
\label{def-hat-p}
	\hat p := \max \{2,p\},
	\quad \hat{p}' := \min \left\{2,\tfrac{p}{p-1} \right\},
	\quad \hat p_* := \max\left\{ 1, \tfrac{n\hat p}{n+2} \right\}.
\end{equation}
Moreover, we define the parabolic function spaces
$$
	V^p(E) := \left\{ u \in L^p (E, \R^N) : Du \in L^p(E,\R^{Nn}) \right\}
$$
and
\begin{equation}\label{eq:V2p}
	V^p_2 (E) := V^p(E) \cap L^2(E,\R^N),
\end{equation}
and consider the subspaces
\begin{equation*}
    \mathcal{V}^{p,0}(E) :=
    \Big\{u\in V^{p}(E) : u(t) \in W^{1,p}_0(E^t,\R^N) \mbox{ for a.e.~$t\in[0,T)$} \Big\}
\end{equation*}
and
\begin{equation*}
    \mathcal{V}^{p,0}_2(E) :=
    \Big\{u\in V^{p}_2(E) : u(t) \in W^{1,p}_0(E^t,\R^N) \mbox{ for a.e.~$t\in[0,T)$} \Big\}.
\end{equation*}
We equip these spaces with the norms
\begin{equation*}
\|v\|_{V^p(E)} := \| Dv \|_{L^p(E,\R^{Nn})} + \|v\|_{L^p(E, \R^N)}
\end{equation*}
and
\begin{equation*}
\|v\|_{V^p_2(E)} := \| Dv \|_{L^p(E,\R^{Nn})} + \|v\|_{L^{\hat p}(E, \R^N)},
\end{equation*}
respectively.
As usual, $(\mathcal{V}^{p,0}(E))^{\prime}$ and $(\mathcal{V}^{p,0}_2(E))^{\prime}$ equipped with the respective operator norms denote the dual spaces of $\mathcal{V}^{p,0}(E)$ and $\mathcal{V}^{p,0}_2(E)$.
For the lateral and initial boundary data, we define the spaces 
\begin{align}
	\label{def:G-sigma}
	\mathfrak{G}^{\sigma} =
	\big\{g \in L^{ p} (E, \R^N) : 
	& \;D g \in L^{p+\sigma}(E,\R^{Nn}), g\chi_E \in L^\infty(0,T;L^2(\R^n,\R^N)) \\
	&\text{ and } \partial_t g \in L^{\hat{p}'(1+\frac{\sigma}{p})}(E,\R^N) \big\},
	\nonumber
\end{align}
and
\begin{align}
	\mathfrak{G}_o^{\sigma} = W^{1,\hat p_* + \sigma}(E^0, \R^N) \cap L^2(E^0, \R^N),
	\label{def:G-sigma-o}
\end{align}
respectively, where $\sigma \geq 0$.
We consider lateral data $g\in\mathfrak{G}^{\sigma}$ and initial data
$g_o\in\mathfrak{G}_o^{\sigma}$ that are compatible in the sense 
\begin{equation} \label{compatibility-lateral-initial}
\bint_{0}^h \int_{\R^n} |g - g_o|^2 \chi_E \, \d x \d t \xrightarrow{h \to 0} 0.
\end{equation}

Now, we are ready to define weak solutions to \eqref{eq:pde}.

\begin{definition} \label{def:weak-sol-global}
Suppose that the vector field $\A \colon E \times \R^N \times \R^{Nn} \to \R^{Nn}$ satisfies~\eqref{assumption:A} and $F \in L^p(E,\R^{Nn})$. We call a map $u \in V^p_2(E)$, with $u \chi_E \in L^\infty(0,T;L^2(\R^n,\R^N))$, a weak solution to~\eqref{eq:pde} with lateral boundary values $g \in \mathfrak{G}^0$ and initial boundary values $g_o \in \mathfrak{G}^0_o$ if 
$$
(u-g)(t)\in W^{1,p}_0(E^t,\R^N)\quad \mbox{ for a.e.~$t\in(0,T)$},
$$
\begin{equation} \label{eq:initial-cond}
\bint_{0}^h \int_{\R^n} |u - g_o|^2 \chi_E \, \d x \d t \xrightarrow{h \to 0} 0,
\end{equation}
and
\begin{equation} \label{eq:weak-pde}
	\iint_{E}  u\cdot \partial_t \varphi - \A (x,t,u,Du) \cdot D \varphi  \, \d x \d t
	=
	\iint_{E} |F|^{p-2}F \cdot D \varphi \, \d x \d t
\end{equation}
for every $\varphi \in C_0^\infty(E, \R^N)$.
\end{definition}

\begin{remark}
\label{rem:time-derivative}
  We note that under the assumptions of Theorem~\ref{thm:main},
  every weak solution $u$ to~\eqref{eq:pde} in the sense of
  Definition~\ref{def:weak-sol-global} satisfies $\partial_tu\in
  (\mathcal{V}^{p,0}(E))'$.
  To verify this claim, we consider a weak solution $u$ in the above sense and a test function
  $\varphi\in C^\infty_0(E,\R^N)$. Equation~\eqref{eq:weak-pde}, assumption~\eqref{assumption:A} and
  H\"older's inequality imply 
  \begin{align*}
    |\langle\partial_tu,\varphi\rangle|
    &=
      \bigg|\iint_Eu\cdot\partial_t\varphi\,\dx\dt\bigg|
    =
    \bigg|\iint_E\big(\A(x,t,u,Du)+|F|^{p-2}F\big)\cdot D\varphi\,\dx\dt\bigg|\\
    &\le
      \iint_E \big(C_1|Du|^{p-1}+|F|^{p-1}\big)|D\varphi|\,\dx\dt\\
    &\le
      \big(C_1\|Du\|_{L^p(E)}^{p-1}+\|F\|_{L^p(E)}^{p-1}\big)\|\varphi\|_{V^{p}(E)}. 
  \end{align*}
  Consequently, the distributional time derivative
  $\partial_tu \colon C^\infty_0(E,\R^N)\to\R$ is continuous with respect to
  the $V^{p}(E)$-norm.  Moreover, since $\R^n\setminus E^t$
  is uniformly $p$-fat in the sense of Definition~\ref{def:unif-p-fat}
  for every $t\in(0,T)$ and~\eqref{assumption:two-sided-growth} holds
  true, Lemma~\ref{lem:c0infty-v2p-density} implies
  that $C^\infty_0(E,\R^N)$ is dense in
  $\mathcal{V}^{p,0}(E)$ with respect to the $V^{p}(E)$-norm. Therefore, the time derivative can be
  extended in a unique way to a bounded linear operator
  $\partial_tu \colon \mathcal{V}^{p,0}(E)\to\R$. In this sense, we have
  $\partial_tu\in (\mathcal{V}^{p,0}(E))'$.
\end{remark}

\subsection{Further notation}
Later on, we will use the following notation.
For a point $z_o = (x_o,t_o) \in \R^n \times \R$ we consider cylinders in space-time of the form
$$
	Q_{r,s}(z_o) :=
	B_r(x_o) \times \Lambda_s(t_o),
$$
where $B_r(x_o) \subset \R^n$ denotes the ball with center $x_o$ and radius $r>0$, and
$$
	\Lambda_s(t_o) :=
	(t_o-s,t_o+s);
$$
further, for $\lambda>0$ intrinsic cylinders are defined by
$$
	Q_\rho^{(\lambda)}(z_o) :=
	B_\rho(x_o) \times \Lambda_\rho^{(\lambda)}(t_o),
$$
where
$$
	\Lambda_\rho^{(\lambda)}(t_o) :=
	(t_o - \lambda^{2-p} \rho^2, t_o + \lambda^{2-p} \rho^2).
$$
For the sake of simplicity, we will omit $z_o$, $x_o$ and $t_o$ in our notation when it is clear from the context.

Next, for sets $A \subset \R^n$ and $U \subset \R^{n+1}$ and functions $u$ and $v$ defined in $A$ and $U$, respectively, we write
$$
	\bint_A u \,\dx
	:=
	\frac{1}{|A|} \int_A u \,\dx
	\qquad \text{and} \qquad
	\biint_U v \,\dx\dt
	:=
	\frac{1}{|U|} \iint_U v \,\dx\dt,
$$
provided that $|A|, |U| > 0$, where $|\cdot|$ is used for the Lebesgue measures both in $\R^n$ and $\R^{n+1}$.
If $A= B_r(x_o)$ is a ball, we also denote the preceding integral average by $(u)_{x_o,r}$.

\subsection{Auxiliary results}
In this section, we collect auxiliary result that are well-known in the literature.
General references for nonlinear potential theory are \cite{Heinonen-etal-book,KLV21}.
We start with the following consequence of $p$-fatness, see \cite[Lemma 3.8]{Parviainen-degenerate}.
\begin{lemma} \label{lem:consequence-fatness}
Let $\Omega \subset \R^n$ be a bounded open set and suppose that $\R^n\setminus \Omega$ is uniformly $p$-fat in the sense of Definition~\ref{def:unif-p-fat}. Let $y\in \Omega$ such that $B_{\rho/3}(y)\setminus \Omega \neq \varnothing$. Then, there exists a constant $c =c(n,p,\alpha)>0$ such that 
$$
\ca_p\big(\overline B_{\rho/2}(y)\setminus \Omega,B_{\rho}(y)\big) \geq  c \ca_p \big(\overline B_{\rho/2}(y),B_{\rho}(y)\big).
$$
\end{lemma}

Moreover, we have the following estimates for the variational $p$-capacity of a smaller ball in a concentric larger ball, see \cite[Lemma 5.35]{KLV21}.
\begin{lemma} \label{lem:estimates-p-capacity}
Let $x_o \in \R^n$, $\rho>0$.
Then we have that
$$
	c(n,p) \rho^{n-p}
	\leq
	\ca_p \big(\overline B_{\rho/2}(y),B_{\rho}(y)\big)
	\leq
	c(n) \rho^{n-p}.
$$
\end{lemma}

Next, if a set is uniformly $p$-fat, then it is also uniformly fat with respect to any larger exponent $\vartheta$, see \cite[Remark 6.20]{KLV21}.
In particular, any nonempty closed set $A \subset \R^n$ satisfies the $\vartheta$-fatness condition for any $\vartheta \in (n,\infty)$, see \cite[Lemma 6.19]{KLV21}.
\begin{lemma} \label{lem:fatness-larger-exponents}
If a compact set $A$ is uniformly $p$-fat with fatness constant $\alpha>0$, then $A$ is uniformly $\vartheta$-fat for any $\vartheta \geq p$ with fatness constant $\alpha_\vartheta = \alpha_\vartheta(n,p,\vartheta,\alpha)$.
\end{lemma}

In contrast, the uniform $p$-fatness of a closed set $A \subset \R^n$ does not imply the uniform fatness with respect to arbitrary smaller exponents.
However, we have the following self-improving property, which was first obtained by Lewis \cite{Lewis} in the more general context of Riesz potentials.
For a direct proof of the following statement, see \cite[Theorem 7.21]{KLV21}.
\begin{theorem} \label{thm:fatness-improvement}
Let $1<p<\infty$. If a set $A \subset \R^n$ is uniformly $p$-fat with fatness constant $\alpha>0$, then there exists $\gamma=\gamma(n,p,\alpha)\in (1,p)$ for which $A$ is uniformly $\gamma$-fat with fatness constant $\alpha_\gamma = \alpha_\gamma(n,p,\alpha)$.
\end{theorem}

The following Maz'ya type inequality links the notion of capacity to the boundary Poincaré inequality.
For its proof, we refer to \cite[Chapter 10]{Mazya-book}, see also \cite[Theorem 5.47]{KLV21}.
Note that the quasicontinuity of the chosen representative (see e.g.~\cite[Definition 5.12]{KLV21} for the notion of $p$-quasicontinuity) cannot be omitted, since the claim is false for arbitrary Sobolev functions.
However, every Sobolev function has a quasicontinuous representative, see \cite[Theorem 5.14]{KLV21}.
\begin{lemma} \label{lem:mazya-sobolev}
Let $B_\rho(x_o)$ be a ball in $\R^n$ and fix a $q$-quasicontinuous
representative of $u \in W^{1,q}(B_\rho(x_o))$.  Denote 
$$
N_{B_{\rho/2}(x_o)}(u):=\{x\in \overline B_{\rho/2}(x_o):u(x)=0\}.
$$
Then, for $\tilde q \in[q,q^\ast]$ with $q^\ast=\frac{nq}{n-q}$ there exists a constant $c=c(n,q)>0$ such that
$$
\left(\mint_{B_\rho(x_o)} |u|^{\tilde q} \d x \right)^{\frac{1}{\tilde q}}\leq  \left( \frac{c}{\ca_q(N_{B_{\rho/2}(x_o)}(u),B_\rho(x_o))} \int_{B_\rho(x_o)} |Du|^q \d x\right)^{\frac 1 q}.
$$
\end{lemma}

At this point, we recall the Gagliardo--Nirenberg inequality, see \cite{Nirenberg}.
We give the statement in the following form that is useful for our purposes.
\begin{lemma} \label{lem:GN}
Let $1 \leq \sigma,q, r < \infty$ and $\theta \in (0,1)$ such that $-\frac{n}{\sigma} \leq \theta (1- \frac{n}{q}) - (1-\theta) \frac{n}{r}$. Then, there exists a constant $c = c(n,\sigma)$ such that for any ball $B_\rho(x_o) \subset \R^n$ with $\rho > 0$ and any function $u \in W^{1,q}(B_\rho(x_o))$ we have
$$
\bint_{B_\rho(x_o)} \frac{|u|^\sigma}{\rho^\sigma} \, \d x \leq c \left[ \bint_{B_\rho(x_o)} \left( \frac{|u|^q}{\rho^q} + |Du|^q \right) \, \d x \right]^\frac{\theta \sigma}{q} \left[ \bint_{B_\rho(x_o)} \frac{|u|^r}{\rho^r} \, \d x \right]^\frac{(1-\theta) \sigma}{r}.
$$
\end{lemma}

The next iteration lemma follows from \cite[Lemma 6.1]{Giusti:book}.
\begin{lemma}
\label{lem:iteration} 
Let $0<\theta<1$, $A,C\geq 0$ and $\beta>0$. Then, there
exists a constant $c=c(\beta,\theta)$ such that there holds: For any $0<r<\rho$ and any non-negative bounded function $\phi\colon [r,\rho] \to \R_{\geq 0}$ satisfying
$$
\phi(t) \leq \theta \phi(s) +A(s-t)^{-\beta}+C \quad \text{ for all } r\leq t <s \leq \rho,
$$
we have
$$
\phi(r) \leq c \left[A(\rho-r)^{-\beta}+C \right].
$$
\end{lemma}

\subsection{Integration by parts formula and density of smooth functions}
First, we recall the following density result from~\cite[Proposition 3.19]{SSSS-nondecreasing}.

\begin{lemma} \label{lem:c0infty-v2p-density}
Suppose that $1 < p < \infty$, $E$ satisfies~\eqref{assumption:two-sided-growth} and $\R^n \setminus E^t$ is uniformly $p$-fat, according to Definition~\ref{def:unif-p-fat}, for every $t \in [0,T)$. Then $C_0^\infty(E, \R^N)$ is dense in $\mathcal{V}^{p,0}(E)$ and $\mathcal{V}_2^{p,0}(E)$, respectively.
\end{lemma}

\begin{remark}
Observe that~\eqref{assumption:two-sided-growth} is not the weakest possible condition under which Lemma~\ref{lem:c0infty-v2p-density} holds true. It suffices to assume that the domain $E$ does not shrink in a discontinuous way (see~\cite[(3.17)]{SSSS-nondecreasing} for the precise formulation of the condition). 
\end{remark}

Next, we obtain an integration by parts formula by considering the special case $q=1$ and $v \equiv 0$ in \cite[Corollary 6.3]{SSSS-general}.
While it has been formulated for $n \geq 2$ in \cite{SSSS-general}, it also holds for $n=1$.
\begin{proposition}\label{prop:ineq_integration_by_parts_formula}
Let $p> \frac{2(n+1)}{n+2}$, and assume that the domain $E$ satisfies \eqref{assumption:one-sided-growth} with $\ell$ given by \eqref{assumption:function-one-sided-growth}, and for all $t \in [0,T)$, the complement of the time slice $E^t$ is uniformly $p$-fat with a parameter $\alpha>0$.
Further, consider a function $u\in \mathcal{V}^{p,0}(E)\cap L^{\infty}(0,T;L^{2}(\Omega, \R^{N}))$ with time derivative $\partial_t u \in (\mathcal{V}^{p,0}_{2}(E))'$. Then, the formula
\begin{align*}
\langle \partial_t u , \zeta u \rangle \geq - \tfrac12 \iint_E \zeta' |u|^2 \, \d x\d t
\end{align*}
holds for any non-negative function $\zeta \in C_0^{0,1}((0,T))$.
\end{proposition}

\section{Energy
  estimate near the lateral boundary}
\label{sec:Caccioppoli}
In this section, we establish energy estimates, also called
Caccioppoli estimates for cylinders near the lateral boundary of
$E$. The interior case and the case of cylinders close to the initial
boundary that do not touch the lateral boundary
are analogous to the case of cylindrical domains, see \cite{Boegelein:1,BP}.

The following energy estimate is well-known in cylindrical domains, see e.g.~\cite[Lemma 3.2]{Parviainen-degenerate}.
In the noncylindrical situation, a key ingredient of the proof is the integration by parts formula in Lemma \ref{prop:ineq_integration_by_parts_formula}.
\label{sec:Caccioppoli-lateral}
\begin{lemma} \label{lem:Caccioppoli-boundary} 
Let $p>\frac{2(n+1)}{n+2}$ and assume that $u$ is a weak solution according to Definition~\ref{def:weak-sol-global} with lateral and initial boundary values $g\in\mathfrak{G}^{0}$ and 
$g_o\in\mathfrak{G}_o^{0}$ that satisfy~\eqref{compatibility-lateral-initial}. Suppose that $E$ satisfies \eqref{assumption:one-sided-growth} with $\ell$ given by \eqref{assumption:function-one-sided-growth}, and \eqref{assumption:two-sided-growth}, and that $\R^n \setminus E^t$ is uniformly $p$-fat for every $t \in [0,T)$. Let $Q_{R,S}(z_o) \subset \R^{n+1}$ with $0<r< R \leq 1$ and $0 < s < S \leq 1$. Then, there exists $c = c(C_o,C_1,p) > 0$ such that
\begin{align*}
\sup_{t\in \Lambda_s(t_o) \cap (0,T)} &\int_{B_r(x_o) \cap E^t} |u(t) - g|^2\, \d x + \iint_{Q_{r,s}(z_o) \cap E} |Du|^p \, \d x \d t \\
& \leq  c \iint_{Q_{R,S}(z_o) \cap E} \frac{|u-g|^p}{(R-r)^p} + \frac{|u-g|^2}{S-s} + G \, \d x \d t
\end{align*}
holds true.
\end{lemma}

\begin{proof}
For $t_1 \in \Lambda_{s}(t_o) \cap (0,T)$ we define
\begin{align*}
	\psi_\varepsilon(t) := \left\{
	\begin{array}{cl}
		\frac{t-\varepsilon}{\varepsilon},& \text{for } t\in (\varepsilon,2\varepsilon], \\[5pt]
		1, & \text{for } t\in (2\varepsilon,t_1-2\varepsilon],\\[5pt]
		\frac{t_1-\varepsilon-t}{\varepsilon},& \text{for } t \in (t_1-2\varepsilon,t_1-\varepsilon],\\[5pt]
		0,& \text{otherwise.} 
	\end{array}
	\right.
\end{align*}
Let $\tilde{\eta} \colon \R \to [0,1]$ be given by $\tilde \eta (\sigma) = \min\{ \max \{\sigma,0\}^{\max\{2/p,1\}},1 \}$ and set $\eta \colon \R^n \to [0,1]$,
$$
	\eta(x)
	:=
	\tilde{\eta} \left( 1 - \frac{\dist(x, B_r(x_o))}{R-r} \right).
$$
Then, we have that $\eta^p, \eta^\frac{p}{2} \in W^{1,\infty}(B_R(x_o);[0,1])$ with $\eta \equiv 0$ outside of $B_R(x_o)$, $\eta \equiv 1$ in $B_r(x_o)$, and $|D\eta|\leq \frac{c}{R-r}$. 
Further, let $\zeta \in W^{1,\infty}\left( \Lambda_S(t_o),[0,1] \right)$ be defined by
\begin{align*}
	\zeta(t):= \left\{
	\begin{array}{cl}
		\frac{t-t_o+S}{S-s}, & \text{for } t \in (t_o-S,t_o-s), \\[5pt]
		1,& \text{for } t \geq t_o-s.
	\end{array}
	\right.
\end{align*}
Now, note that $\partial_t u \in (\mathcal{V}^{p,0}(E))'$ by Remark \ref{rem:time-derivative}.
Rewriting \eqref{eq:weak-pde} as
\begin{equation} \label{eq:Caccioppoli-aux}
	\langle \partial_t u, \varphi \rangle
	+ \iint_E \mathbf{A}(x,t,u,Du) \cdot D\varphi \,\dx\dt
	=
	- \iint_E |F|^{p-2} F \cdot D\varphi \,\dx\dt,
\end{equation}
and using that $C^\infty_0(E,\R^N)$ is dense in $\mathcal{V}^{p,0}_2(E)$ by Lemma \ref{lem:c0infty-v2p-density}, and that $u-g \in \mathcal{V}^{p,0}_2(E)$ by Definition~\ref{def:weak-sol-global}, we may choose
$$
	\varphi(x,t)
	:=
	\eta^p(x) \zeta(t) \psi_\varepsilon(t) (u(x,t)-g(x,t))
$$
as a testing function in \eqref{eq:Caccioppoli-aux}.
For the parabolic part, i.e.~the first term on the left-hand side of \eqref{eq:Caccioppoli-aux}, we obtain that
\begin{align*}
	\big\langle \partial_t u , \eta^p \zeta \psi_\varepsilon (u-g) \big\rangle
	=
	\big\langle \partial_t (u-g) , \eta^p \zeta \psi_\varepsilon (u-g) \big\rangle
	+ \big\langle \partial_t g , \eta^p \zeta \psi_\varepsilon (u-g) \big\rangle
	=: \mathrm{I} + \mathrm{II},
\end{align*}
where the definition of $\mathrm{I}$ and $\mathrm{II}$ is clear from the context.
By Proposition~\ref{prop:ineq_integration_by_parts_formula} applied to
$\eta^\frac{p}{2}(u-g)$, assumption \eqref{compatibility-lateral-initial}, and the initial condition~\eqref{eq:initial-cond}, we find that 
\begin{align*}
\mathrm I &\geq - \tfrac12 \iint_{Q_{R,S} \cap E} \eta^p |u-g|^2 (\psi_\eps'\zeta + \zeta' \psi_\epsilon) \, \d x \d t \\
&\xrightarrow{\eps \downarrow 0} \tfrac12 \int_{B_R \cap E^{t_1}} \eta^p |u-g|^2 \zeta \, \d x - \tfrac12 \iint_{Q_{R,S} \cap E}  \eta^p \zeta'|u-g|^2 \, \d x\d t.
\end{align*}
On the one hand, in the case $p > 2$ we use Young's inequality and the fact that $R-r \leq 1$ to get that
\begin{align*}
\mathrm{II} &= \iint_{Q_{R,S} \cap E}  \eta^p\zeta\psi_\varepsilon \partial_t g \cdot (u-g) \, \d x\d t \\
&\geq - \iint_{Q_{R,S} \cap E} \eta^p\zeta\psi_\varepsilon |\partial_t g|^{p'} \, \d x\d t - c(p) \iint_{Q_{R,S} \cap E} \eta^p\zeta\psi_\varepsilon \frac{|u-g|^p}{(R-r)^p} \, \d x\d t.
\end{align*}
On the other hand, if $\frac{2(n+1)}{n+2} \leq p<2$, Young's inequality and the fact that $S -s \leq 1$ yield
\begin{align*}
\mathrm{II} &\geq - \iint_{Q_{R,S} \cap E} \eta^p\zeta\psi_\varepsilon |\partial_t g|^{2} \, \d x\d t - c \iint_{Q_{R,S} \cap E} \eta^p\zeta\psi_\varepsilon \frac{|u-g|^2}{S-s} \, \d x\d t.
\end{align*}
For the divergence part, i.e.~the second term on the left-hand side of \eqref{eq:Caccioppoli-aux}, by the structural conditions~\eqref{assumption:A} and Young's inequality, we conclude that
\begin{align*}
\iint_{E} &\psi_\eps \zeta \mathbf{A}(x,t,u,Du ) \cdot D [\eta^p (u-g)] \, \d x \d t \\
&= \iint_{Q_{R,S} \cap E} \psi_\eps \zeta \eta^p  \mathbf{A}(x,t,u,Du ) \cdot D (u-g) \, \d x \d t \\
&\phantom{=} + p \iint_{Q_{R,S} \cap E} \psi_\eps \zeta \eta^{p-1} (u-g) \mathbf{A}(x,t,u,Du ) \cdot D \eta \, \d x \d t \\
&\geq C_o \iint_{Q_{R,S} \cap E} \psi_\eps \zeta \eta^p |D u|^p \, \d x \d t - C_1 \iint_{Q_{R,S} \cap E} \psi_\eps \zeta \eta^p  |Du|^{p-1} |Dg| \, \d x \d t \\
&\phantom{=} - C_1 p \iint_{Q_{R,S} \cap E} \psi_\eps \zeta \eta^{p-1} |u-g| |Du|^{p-1} |D \eta| \, \d x \d t \\
&\geq \tfrac{C_o}{2} \iint_{Q_{R,S} \cap E} \psi_\eps \zeta \eta^p |D u|^p \, \d x \d t - c(C_o,C_1,p)\iint_{Q_{R,S} \cap E} |Dg|^p \, \d x \d t  \\
&\phantom{=} - c(C_o,C_1,p) \iint_{Q_{R,S} \cap E} \frac{|u-g|^p}{(R-r)^p}\, \d x\d t.
\end{align*}
For the term on the right-hand side of \eqref{eq:Caccioppoli-aux}, Young's inequality gives us that
\begin{align*}
	- \iint_{E} &\psi_\eps \zeta |F|^{p-2}F \cdot D[\eta^p(u-g)] \, \d x\d t \\
	&\leq \tfrac{C_o}{4} \iint_{Q_{R,S} \cap E}  \psi_\eps \zeta \eta^p |Du|^p \, \d x\d t
	\\ &\phantom{=}
	+ c(C_o,p) \iint_{Q_{R,S} \cap E} \frac{|u-g|^p}{(R-r)^p} + |Dg|^p + |F|^p \, \d x\d t.
\end{align*}
Inserting the preceding estimates into \eqref{eq:Caccioppoli-aux} and passing to the limit $\eps \downarrow 0$, we infer
\begin{align*}
\int_{B_r(x_o) \cap E^{t_1}} & |u-g|^2 \, \d x + \iint_{( B_r(x_o) \times (t_o - s,t_1) ) \cap E} |Du|^p \, \d x\d t \\
&\leq c\iint_{Q_{R,S}(z_o) \cap E} \frac{|u-g|^p}{(R-r)^p} + \frac{|u-g|^2}{S-s} \, \d x\d t \\
&\quad + c\iint_{Q_{R,S}(z_o) \cap E} |F|^p + |Dg|^p + |\partial_t g|^{\hat p'} \, \d x\d t.
\end{align*}
Noting that both terms on the left-hand side of the preceding inequality are non-negative, taking the supremum over $t_1 \in \Lambda_s$ in the first term, and passing to the limit $t_1 \uparrow t_o + s$ in the second term, we conclude the proof of the lemma.
\end{proof}

\section{Reverse H\"older inequalities}
\label{sec:rev-Holder}
Our goal in this section is to derive reverse Hölder inequalities.
To this end, for $z_o \in E$, $\lambda>0$, $\beta \in \big(\tfrac12,1\big]$, and
\begin{equation}
	0<\varrho \leq \bigg( \frac{\min \left\{1, \lambda^{(p-2)\beta} \right\}}{2^{7\beta-3} M} \bigg)^\frac{1}{2\beta-1}, 
	\label{eq:radius-small}
\end{equation}
we will distinguish between the following cases.
On the one hand, if
\begin{equation}
	B_{8\varrho}(x_o) \setminus E^t \neq \varnothing
	\text{ for some } t \in \Lambda_{8\varrho}^{(\lambda)}(t_o) \cap (0,T),
	\label{eq:case-lateral}
\end{equation}
the cylinder is near the lateral boundary; we refer to Section \ref{sec:rev-Holder-lateral}. On the other hand, if we have that
\begin{equation}
	B_{8\varrho}(x_o) \setminus E^t = \varnothing
	\text{ for all } t \in \Lambda_{8\varrho}^{(\lambda)}(t_o) \cap (0,T)
	\qquad \text{and} \qquad 0 \in \Lambda_{2\varrho}^{(\lambda)}(t_o);
	\label{eq:case-initial}
\end{equation}
we are dealing with cylinders near the initial boundary; see Section~\ref{sec:rev-Holder-initial}.
In Section~\ref{sec:rev-Holder-interior} we consider the remaining case of interior cylinders, i.e., cylinders such that
\begin{equation}
	B_{8\varrho}(x_o) \times \Lambda_{2 \rho}^{(\lambda)}(t_o) \subset E.
	\label{eq:case-interior}
\end{equation}

\subsection{Near the lateral boundary} \label{sec:rev-Holder-lateral}
In this section, we prove a reverse Hölder inequality for cylinders near the lateral boundary of $E$.
We start with the following lemma that shows that if $B_{8\varrho}(x_o)$ intersects the complement of a time slice $E^t$ for some time $s \in \Lambda_{8\varrho}^{(\lambda)}(t_o) \cap (0,T)$, and the radius $\varrho>0$ is small enough, then the larger ball $B_{16\varrho}(x_o)$ intersects $\R^n \setminus E^t$ for all times $t \in \Lambda_{8\varrho}^{(\lambda)}(t_o) \cap (0,T)$.
\begin{lemma} \label{lem:ball-complement-intersect}
Suppose that $E$ satisfies~\eqref{assumption:two-sided-growth}, that $\lambda > 0$, and that the radius $\varrho>0$ satisfies \eqref{eq:radius-small} for some $\beta \in \big(\tfrac12 ,1\big]$.
If \eqref{eq:case-lateral} holds true, then we have that
$$
B_{16 \rho} (x_o) \setminus E^t \neq \varnothing \quad \text{ for every } t \in \Lambda_{8 \rho}^{(\lambda)}(t_o) \cap (0,T).
$$
\end{lemma}

\begin{proof}
Let $s \in \Lambda_{8 \rho}^{(\lambda)}(t_o) \cap (0,T)$ denote the time in \eqref{eq:case-lateral} at which $B_{8\varrho}(x_o) \setminus E^s \neq \varnothing$.
By the bound on $\rho$ and condition~\eqref{assumption:two-sided-growth}, there exists $y \in B_{8 \rho} (x_o) \setminus E^s$ such that
$$
	\dist(y, \R^n \setminus E^t)
	\leq
	\mathbf{e}^c(E^t, E^s)
	\leq
	M\big( 2 \lambda^{2-p} (8 \rho)^2 \big)^\beta \leq 8 \rho
$$
for every $t \in \Lambda_{8 \rho}^{(\lambda)}(t_o)$. Thus, for every $t \in \Lambda_{8\rho}^{(\lambda)}$ we find $z \in \R^n \setminus E^t$ such that $|z-y| \leq 8 \rho$, and further
$$
|z-x_o| \leq |z-y| + |y-x_o| \leq 16 \rho,
$$
which concludes the proof of the lemma.
\end{proof}

Next, we prove the following boundary Poincaré inequality.
\begin{lemma}
\label{lem:boundary-Poincare}
Let $\R^n \setminus E^t$ be uniformly $p$-fat with fatness constant $\alpha>0$, and let $u(t) \in g(t) + W^{1,p}_0(E^t,\R^N)$.
Moreover, assume that for $x_o \in E^t$ and $r>0$ we have that $B_\frac{r}{3}(x_o) \setminus E^t \neq \varnothing$.
Then, there exists an exponent $\gamma = \gamma (n,p,\alpha)$ such that for every $\gamma \leq \vartheta \leq p$
$$
	\int_{B_r(x_o) \cap E^t} \bigg| \frac{u-g}{r} \bigg|^\vartheta(t) \,\dx
	\leq
	c \int_{B_r(x_o) \cap E^t} |D(u-g)|^\vartheta(t) \,\dx
$$
holds true with a constant $c=c(n,p,\vartheta,\alpha)$.
\end{lemma}

\begin{proof}
By Theorem \ref{thm:fatness-improvement} and Lemma \ref{lem:fatness-larger-exponents}, there exists $\gamma = \gamma (n,p,\alpha)$ such that $\R^n \setminus E^t$ is uniformly $\vartheta$-fat for any $\gamma \leq \vartheta \leq p$ with the fatness constant $\alpha_\vartheta = a_\vartheta(n,p,\vartheta,\alpha)$.
Extending $u-g \in W^{1,\vartheta}_0(E^t,\R^N)$ to $B_r(x_o) \setminus E^t$ by zero, and choosing a $\vartheta$-quasicontinuous representative, by Lemma \ref{lem:mazya-sobolev} we find that
\begin{align}
	\int_{B_r(x_o) \cap E^t} &|u-g|^\vartheta(t) \,\dx
	\nonumber \\ &\leq
	\frac{c(n,\vartheta) r^n}{\operatorname{cap}_\vartheta\big( N_{B_{r/2}(x_o)}(u-g), B_r(x_o) \big)} \int_{B_r(x_o) \cap E^t} |D(u-g)|^\vartheta(t) \,\dx,
	\label{eq:boundary-Poincare-aux}
\end{align}
where
$$
	N_{B_{r/2}(x_o)}(u-g) :=
	\big\{ x \in \overline{B_{r/2}(x_o)} : u(x) - g(x) = 0 \big\}.
$$
Note that according to \cite[Theorem 5.26]{KLV21}, the representative
of $u-g$ can be chosen such that $u-g=0$ $p$-quasieverywhere in $\R^n \setminus E^t$ (i.e.~everywhere outside of a set of $p$-capacity zero).
Further, we know that $B_\frac{r}{3}(x_o) \setminus E^t \neq \varnothing$ and that $\R^n \setminus E^t$ is uniformly $p$-fat.
Therefore, by Lemma~\ref{lem:consequence-fatness} and Lemma~\ref{lem:estimates-p-capacity}, we obtain that
\begin{align*}
	\ca_\vartheta\big( N_{B_{r/2}(x_o)}(u-g), B_r(x_o) \big)
	&\geq
	\ca_\vartheta\big( \big( \R^n \setminus E^t \big) \cap \overline{B_{r/2}(x_o)}, B_r(x_o) \big)
	\\ &\geq
	c(n,\vartheta,\alpha_\vartheta) \ca_\vartheta\big( \overline{B_{r/2}(x_o)}, B_{r}(x_o) \big)
	\\ &\geq
	c(n,p,\vartheta,\alpha) r^{n-\vartheta}.
\end{align*}
Inserting this into \eqref{eq:boundary-Poincare-aux}, we conclude the proof of the lemma.
\end{proof}

In the singular range $1<p<2$, the following lemma helps us to estimate a quadratic term.
Its proof follows along the lines of \cite[Lemma 4.3]{BP}.
\begin{lemma}
\label{lem:quadratic-term-singular}
Let $\frac{2(n+1)}{n+2} <p < 2$, and let $u$ be a weak solution according to Definition~\ref{def:weak-sol-global} with lateral and initial boundary values $g \in \mathfrak{G}^{0}$ and
$g_o \in \mathfrak{G}^{0}_o$ that satisfy~\eqref{compatibility-lateral-initial}.
Suppose that $E$ satisfies \eqref{assumption:one-sided-growth} with $\ell$ given by \eqref{assumption:function-one-sided-growth}, and \eqref{assumption:two-sided-growth}, and that $\R^n \setminus E^t$ is uniformly $p$-fat for every $t \in [0,T)$.
Further, assume that for $\lambda>0$ and $\varrho>0$ we have that
$B_{16 \rho} (x_o) \setminus E^t \neq \varnothing$ for every $t \in
\Lambda_{8 \rho}^{(\lambda)}(t_o) \cap (0,T)$, and that
\begin{equation}
	\frac{1}{\babs{Q_{48\rho}^{(\lambda)}(z_o)}} \iint_{Q_{48\rho}^{(\lambda)}(z_o) \cap E} |D u|^p + G \, \d x \d t
	\leq \lambda^p.
	\label{eq:scaling-quadratic-term-singular}
\end{equation}
Then, there exists a constant $c=c(n,p,C_o,C_1)$ such that
$$
	\frac{1}{\babs{Q_{4\rho}^{(\lambda)}(z_o)}} \iint_{Q_{4\rho}^{(\lambda)}(z_o) \cap E} \bigg| \frac{u-g}{\varrho} \bigg|^2 \,\dx\dt
	\leq c \lambda^2.
$$
\end{lemma}

\begin{proof}
Fix $1 \leq \alpha_1 < \alpha_2 \leq 2$.
Extending $u-g$ outside of $E$ by zero, and applying Lemma
\ref{lem:GN} slice-wise with $(\sigma,q,\theta,r) = \big(
2,p,\frac{p}{2},2 \big)$, which is possible since $p>\frac{2n}{n+2}$,
we obtain that
\begin{align}
	\iint_{Q_{4\alpha_1\rho}^{(\lambda)}(z_o) \cap E} &|u-g|^2 \,\dx\dt
	\nonumber \\ &\leq
	c \varrho^2 \int_{\Lambda_{4\alpha_1 \varrho}^{(\lambda)}(t_o)\cap(0,T)}
	\Bigg( \int_{B_{4\alpha_1\varrho}(x_o)} |D(u-g)|^p(t) + \bigg| \frac{u-g}{\varrho} \bigg|^p(t) \,\dx \Bigg)
	\nonumber \\ &\phantom{=} \qquad\qquad\qquad\qquad
	\cdot \Bigg( \bint_{B_{4\alpha_1\varrho}(x_o)} \bigg| \frac{u-g}{\varrho} \bigg|^2(t) \,\dx \Bigg)^\frac{2-p}{2} \,\dt
	\nonumber \\ &\leq
	c \varrho^p \Bigg(
                       \iint_{Q_{8\rho}^{(\lambda)}(z_o)\cap E} |D(u-g)|^p + \bigg| \frac{u-g}{\varrho} \bigg|^p \,\dx\dt \Bigg)
	\label{eq:quadratic-term-aux1}	
	\\ &\phantom{=}
	\cdot \Bigg( \sup_{t \in \Lambda_{4\alpha_1 \varrho}^{(\lambda)}(t_o)\cap(0,T)}
	\bint_{B_{4\alpha_1\varrho}(x_o)} |u-g|^2(t) \,\dx \Bigg)^\frac{2-p}{2}.
	\nonumber
\end{align}
Now, we estimate the terms on the right-hand side of \eqref{eq:quadratic-term-aux1} separately.
First, since we have assumed that $B_{16 \rho} (x_o) \setminus E^t \neq \varnothing$ for every $t \in \Lambda_{8 \rho}^{(\lambda)}(t_o) \cap (0,T)$, by applying Lemma \ref{lem:boundary-Poincare} with $r=48\varrho$ slice-wise, and using \eqref{eq:scaling-quadratic-term-singular}, we find that
\begin{align}
	\iint_{Q_{8\rho}^{(\lambda)}(z_o)\cap
                       E} &|D(u-g)|^p + \bigg| \frac{u-g}{\varrho} \bigg|^p \,\dx\dt
	\nonumber \\ &\leq
	\iint_{Q_{48\rho}^{(\lambda)}(z_o)\cap E} |D(u-g)|^p + \bigg| \frac{u-g}{\varrho} \bigg|^p \,\dx\dt
	\nonumber \\ &\leq
	c \iint_{Q_{48\rho}^{(\lambda)}(z_o)\cap E} |D(u-g)|^p \,\dx\dt
	\nonumber \\ &\leq
	c \big|Q_{48\varrho}^{(\lambda)} \big| \lambda^p.
	\label{eq:quadratic-term-aux2}
\end{align}
Next, we use the energy estimate in Lemma
\ref{lem:Caccioppoli-boundary}, which holds since
$p>\frac{2(n+1)}{n+2}$. Moreover, we use the fact that $\alpha_2^2 - \alpha_1^2 \geq (\alpha_2 - \alpha_1)^2$, Young's inequality with exponents $\frac{2}{p}$ and $\frac{2}{2-p}$ (which is possible, since $1<p<2$), and \eqref{eq:scaling-quadratic-term-singular} to conclude that
\begin{align}
	&\sup_{t \in \Lambda_{4\alpha_1 \varrho}^{(\lambda)}(t_o)\cap(0,T)}
	\bint_{B_{4\alpha_1\varrho}(x_o)} |u-g|^2(t) \,\dx
	\nonumber \\ &\leq
	\frac{c}{\big| B_{4\alpha_1\varrho(x_o)} \big|}
	\iint_{Q_{4\alpha_2\rho}^{(\lambda)}(z_o)\cap
                       E} \bigg| \frac{u-g}{(\alpha_2-\alpha_1)\varrho} \bigg|^p
	+ \lambda^{p-2} \bigg| \frac{u-g}{(\alpha_2-\alpha_1)\varrho} \bigg|^2 + G \,\dx\dt
	\nonumber \\ &\leq
	\frac{c \lambda^{2-p} \varrho^2}{\big| Q_{4\alpha_2\varrho}^{(\lambda)}(z_o) \big|}
	\iint_{Q_{4\alpha_2\rho}^{(\lambda)}(z_o)\cap
                       E} \lambda^{p-2} \bigg| \frac{u-g}{(\alpha_2-\alpha_1)\varrho} \bigg|^2 + \lambda^p + G \,\dx\dt
	\nonumber \\ &\leq
	\frac{c}{\big| Q_{4\alpha_2\varrho}^{(\lambda)}(z_o) \big|}
	\iint_{Q_{4\alpha_2\rho}^{(\lambda)}(z_o)\cap
                       E} \bigg| \frac{u-g}{\alpha_2-\alpha_1} \bigg|^2 \,\dx\dt
	+ c\lambda^2 \varrho^2.
	\label{eq:quadratic-term-aux3}
\end{align}
Inserting \eqref{eq:quadratic-term-aux2} and \eqref{eq:quadratic-term-aux3} into \eqref{eq:quadratic-term-aux1}, and applying Young's inequality with exponents $\frac{2}{p}$ and $\frac{2}{2-p}$ yields
\begin{align*}
	&\iint_{Q_{4\alpha_1\rho}^{(\lambda)}(z_o) \cap E} |u-g|^2 \,\dx\dt
	\\ &\leq
	c \big|Q_{48\varrho}^{(\lambda)} \big| \lambda^p \varrho^p
	\Bigg( \frac{1}{\big| Q_{4\alpha_2\varrho}^{(\lambda)}(z_o) \big|}
	\iint_{Q_{4\alpha_2\rho}^{(\lambda)}(z_o)\cap
                       E} \bigg| \frac{u-g}{\alpha_2-\alpha_1} \bigg|^2 \,\dx\dt
	+ c\lambda^2 \varrho^2 \Bigg)^\frac{2-p}{2}
	\\ &\leq
	\frac{1}{2} \iint_{Q_{4\alpha_2\rho}^{(\lambda)}(z_o) \cap E} |u-g|^2 \,\dx\dt
	+ c \lambda^2 \varrho^2 \big|Q_{48\varrho}^{(\lambda)} \big| (\alpha_2-\alpha_1)^{-\frac{2(2-p)}{p}},
\end{align*}
for any $1\le\alpha_1<\alpha_2\le 2$. 
Now, by Lemma \ref{lem:iteration} with $\theta = \frac{1}{2}$, $\phi \colon [1,2] \to \R_{\geq 0}$ given by
$$
	\phi(s) :=
	\iint_{Q_{4s\rho}^{(\lambda)}(z_o) \cap E} |u-g|^2 \,\dx\dt,
$$
$A = c \lambda^2 \varrho^2 \big|Q_{48\varrho}^{(\lambda)} \big|$, $\beta = \frac{2(2-p)}{p}$, and $C=0$, we deduce that
$$
	\iint_{Q_{4\rho}^{(\lambda)}(z_o) \cap E} \bigg| \frac{u-g}{\varrho} \bigg|^2 \,\dx\dt
	\leq
	c \big| Q_{48\varrho}^{(\lambda)} \big| \lambda^2
	\leq
	c \big| Q_{4\varrho}^{(\lambda)} \big| \lambda^2.
$$
This concludes the proof of the lemma.
\end{proof}

Now, we are ready to derive the reverse Hölder inequality for cylinders near the lateral boundary.
\begin{lemma} \label{lem:reverse-holder-lateral}
Let $u$ be a weak solution to \eqref{eq:pde} according to Definition~\ref{def:weak-sol-global} with lateral and initial boundary values $g \in \mathfrak{G}^{0}$ and
$g_o \in \mathfrak{G}^{0}_o$ that satisfy~\eqref{compatibility-lateral-initial}.
Assume that $E$ satisfies \eqref{assumption:one-sided-growth} with $\ell$ given by \eqref{assumption:function-one-sided-growth}, and \eqref{assumption:two-sided-growth}, and that $\R^n \setminus E^t$ is uniformly $p$-fat for every $t \in [0,T)$.
Furthermore, consider a cylinder $Q_\rho^{(\lambda)}(z_o) \subset \R^{n+1}$ with $0<\rho \leq 1$, $\lambda \geq 1$ and $\lambda^{2-p} \rho^2 \leq 1$, such that \eqref{eq:radius-small} and \eqref{eq:case-lateral} are satisfied, and suppose that
\begin{equation} \label{eq:boundary-rev-Holder-scaling-assumption}
\frac{1}{\babs{Q_{48\rho}^{(\lambda)}}} \iint_{Q_{48\rho}^{(\lambda)} \cap E} |D u|^p + G \, \d x \d t \leq \lambda^p \leq \frac{1}{\babs{Q_\rho^{(\lambda)}}} \iint_{Q_\rho^{(\lambda)} \cap E} |Du|^p + G \, \d x\d t.
\end{equation}
Then, denoting the exponent from Theorem~\ref{thm:fatness-improvement} by $\gamma = \gamma(n,p,\alpha)$, for every $\max\{\gamma, \hat p_*\} \leq q < p$ there exists a constant $c = c(n,p,C_o,C_1,\alpha) \geq 1$ such that
\begin{align*}
	&\frac{1}{\babs{Q_\rho^{(\lambda)}(z_o)}} \iint_{Q_\rho^{(\lambda)}(z_o) \cap E} |Du|^p\, \d x\d t
	\\ &\qquad\leq \left( \frac{c}{\babs{ Q_{48\rho}^{(\lambda)}(z_o)}}\iint_{Q_{48\rho}^{(\lambda)}(z_o) \cap E} |Du|^q\, \d x\d t \right)^\frac{p}{q} 
	+ \frac{c}{\babs{ Q_{48\rho}^{(\lambda)}(z_o)}} \iint_{Q_{48\rho}^{(\lambda)}(z_o) \cap E} G \, \d x \d t.
\end{align*}
\end{lemma}

\begin{proof}
From the energy estimate in Lemma \ref{lem:Caccioppoli-boundary} we obtain that
\begin{align}
	\frac{1}{\big| Q_\varrho^{(\lambda)}(z_o) \big|} &\iint_{Q_\varrho^{(\lambda)}(z_o) \cap E} |Du|^p \,\dx\dt
	\nonumber \\ &\leq
	\frac{c}{\big| Q_{2\varrho}^{(\lambda)}(z_o) \big|}
	\iint_{Q_{2\varrho}^{(\lambda)}(z_o) \cap E} \bigg| \frac{u-g}{\varrho} \bigg|^p \,\dx\dt
	\label{eq:rev-Holder-boundary-aux1} \\ &\phantom{=}
	+ \frac{c \lambda^{p-2}}{\big| Q_{2\varrho}^{(\lambda)}(z_o) \big|}
	\iint_{Q_{2\varrho}^{(\lambda)}(z_o) \cap E} \bigg| \frac{u-g}{\varrho} \bigg|^2 \,\dx\dt
	\nonumber \\ & \phantom{=}
	+\frac{c}{\big| Q_{2\varrho}^{(\lambda)}(z_o) \big|}
	\iint_{Q_{2\varrho}^{(\lambda)}(z_o) \cap E} G \,\dx\dt.
	\nonumber
\end{align}
In the following, we estimate the terms on the right-hand side of \eqref{eq:rev-Holder-boundary-aux1} separately.
First, for $\sigma \in \{2,p\}$ we write
$$
	\mathrm{I}_\sigma :=
	\frac{\lambda^{p-\sigma}}{\big| Q_{2\varrho}^{(\lambda)}(z_o) \big|}
	\iint_{Q_{2\varrho}^{(\lambda)}(z_o) \cap E} \bigg| \frac{u-g}{\varrho} \bigg|^\sigma \,\dx\dt,
$$
and fix $q \in [ \max\{ \gamma, \hat{p}_\ast \}, p)$.
First, we estimate this term in the case $\sigma\ge p$. We
use $(u-g)(t) \in W^{1,q}(B_{2\varrho}(x_o))$ with $u-g=0$ in
$\R^n \setminus E^t$ $\mathcal{L}^n$-a.e.~for a.e.~$t \in
\Lambda_{2\varrho}^{(\lambda)}(t_o)\cap(0,T)$, and apply Lemma
\ref{lem:GN} with $r=2$ and $\theta = \frac{q}{\sigma}$
slice-wise. Note that this is admissible, since $q \geq \hat{p}_\ast$ and
$\theta<\frac{p}{\sigma}\le1$ in the present case $\sigma\ge p$.
In this way, we conclude that
\begin{align}
	\mathrm{I}_\sigma
	&=
	\frac{\lambda^{p-\sigma}}{\big|
   \Lambda_{2\varrho}^{(\lambda)}(t_o) \big|} \int_{\Lambda_{2\varrho}^{(\lambda)}(t_o)\cap(0,T)}
	\bint_{B_{2\varrho}(x_o)} \bigg| \frac{u-g}{\varrho} \bigg|^\sigma(t) \,\dx \,\dt
	\nonumber \\ &\leq
	\frac{c \lambda^{p-\sigma}}{\big| \Lambda_{2\varrho}^{(\lambda)}(t_o) \big|} \int_{\Lambda_{2\varrho}^{(\lambda)}(t_o)\cap(0,T)}
	\Bigg( \bint_{B_{2\varrho}(x_o)} |D(u-g)|^q(t) + \bigg| \frac{u-g}{\varrho} \bigg|^q(t) \,\dx \Bigg)
	\nonumber \\ &\phantom{=} \qquad\qquad \qquad\qquad\qquad
	\cdot \Bigg( \bint_{B_{2\varrho}(x_o)} \bigg| \frac{u-g}{\varrho} \bigg|^2(t) \,\dx \Bigg)^\frac{\sigma-q}{2} \,\dt
	\nonumber \\ &\leq
	c \lambda^{p-\sigma}
	\Bigg( \frac{c}{\big| Q_{2\varrho}^{(\lambda)}(z_o) \big|} \iint_{Q_{2\varrho}^{(\lambda)}(z_o) \cap E} |D(u-g)|^q + \bigg| \frac{u-g}{\varrho} \bigg|^q \,\dx\dt \Bigg)
	\label{eq:rev-Holder-boundary-aux2}\\ &\phantom{=}\qquad\qquad
	\cdot \Bigg( \sup_{t \in \Lambda_{2\varrho}^{(\lambda)}(t_o)\cap(0,T)} \bint_{B_{2\varrho}(x_o)} \bigg| \frac{u-g}{\varrho} \bigg|^2(t) \,\dx \Bigg)^\frac{\sigma-q}{2}.
	\nonumber
\end{align}
Now, we estimate the terms on the right-hand side of \eqref{eq:rev-Holder-boundary-aux2} separately.
First, note that $B_{16\varrho}(x_o) \setminus E^t \neq \varnothing$ for all $t \in \Lambda_{8\varrho}^{(\lambda)}(t_o) \cap (0,T)$ by \eqref{eq:case-lateral} and Lemma \ref{lem:ball-complement-intersect}.
Therefore, applying Lemma \ref{lem:boundary-Poincare} with $r=48 \varrho$ slice-wise to the first term on the right-hand side of \eqref{eq:rev-Holder-boundary-aux2}, we find that
\begin{align*}
	\frac{c}{\big| Q_{2\varrho}^{(\lambda)}(z_o) \big|} &\iint_{Q_{2\varrho}^{(\lambda)}(z_o) \cap E} |D(u-g)|^q + \bigg| \frac{u-g}{\varrho} \bigg|^q \,\dx\dt
	\\ &\leq
	\frac{c}{\big| Q_{2\varrho}^{(\lambda)}(z_o) \big|} 
	\iint_{ [ B_{48\varrho}(x_o) \times \Lambda_{2\varrho}^{(\lambda)}(t_o) ] \cap E} |D(u-g)|^q + \bigg| \frac{u-g}{\varrho} \bigg|^q \,\dx\dt
	\\ &\leq
	\frac{c}{\big| Q_{2\varrho}^{(\lambda)}(z_o) \big|} 
	\iint_{[ B_{48\varrho}(x_o) \times \Lambda_{2\varrho}^{(\lambda)}(t_o) ] \cap E} |D(u-g)|^q \,\dx\dt
	\\ &\leq
	\frac{c}{\big| Q_{48\varrho}^{(\lambda)}(z_o) \big|} \iint_{Q_{48\varrho}^{(\lambda)}(z_o) \cap E} |D(u-g)|^q \,\dx\dt.
\end{align*}
Next, applying the energy estimate from Lemma \ref{lem:Caccioppoli-boundary} to the second term on the right-hand side of \eqref{eq:rev-Holder-boundary-aux2} yields
\begin{align*}
	\sup_{t \in \Lambda_{2\varrho}^{(\lambda)}(t_o)\cap(0,T)} &\bint_{B_{2\varrho}(x_o)} \bigg| \frac{u-g}{\varrho} \bigg|^2(t) \,\dx 
	\\ &\leq
	\frac{c}{\big| Q_{4\varrho}^{(\lambda)}(z_o) \big|}
	\iint_{Q_{4\varrho}^{(\lambda)}(z_o) \cap E}
	\lambda^{2-p} \bigg| \frac{u-g}{\varrho} \bigg|^p
	+ \bigg| \frac{u-g}{\varrho} \bigg|^2
	+ \lambda^{2-p} G \,\dx\dt.
\end{align*}
In the case $\frac{2(n+1)}{n+2} < p <2 $, we apply Young's inequality
with exponents $\frac{2}{2-p}$ and $\frac{2}{p}$ and then
Lemma~\ref{lem:quadratic-term-singular} to estimate the quadratic term
on the right-hand side and assumption
\eqref{eq:boundary-rev-Holder-scaling-assumption}. This leads to the bound 
\begin{align*}
	\sup_{t \in \Lambda_{2\varrho}^{(\lambda)}(t_o)\cap(0,T)} &\bint_{B_{2\varrho}(x_o)} \bigg| \frac{u-g}{\varrho} \bigg|^2(t) \,\dx 
	\\ &\leq
	\frac{c}{\big| Q_{4\varrho}^{(\lambda)}(z_o) \big|}
	\iint_{Q_{4\varrho}^{(\lambda)}(z_o) \cap E}
             \bigg| \frac{u-g}{\varrho} \bigg|^2
             +\lambda^2
             + \lambda^{2-p} G \,\dx\dt
         \le c\lambda^2.    
\end{align*}
In the case $p \geq 2$, using that $u-g=0$ a.e.~in $\R^{n+1} \setminus E$, Jensen's inequality gives us that
\begin{align*}
	\sup_{t \in \Lambda_{2\varrho}^{(\lambda)}(t_o)\cap(0,T)} &\bint_{B_{2\varrho}(x_o)} \bigg| \frac{u-g}{\varrho} \bigg|^2(t) \,\dx 
	\\ &\leq
	\frac{c \lambda^{2-p}}{\big| Q_{4\varrho}^{(\lambda)}(z_o) \big|}
	\iint_{Q_{4\varrho}^{(\lambda)}(z_o) \cap E}
	\bigg| \frac{u-g}{\varrho} \bigg|^p \,\dx\dt
	\\ &\phantom{=}
	+ \Bigg( \frac{c}{\big| Q_{4\varrho}^{(\lambda)}(z_o) \big|}
	\iint_{Q_{4\varrho}^{(\lambda)}(z_o) \cap E}
	\bigg| \frac{u-g}{\varrho} \bigg|^p \,\dx\dt \Bigg)^\frac{2}{p}
	\\ &\phantom{=}
	+ \frac{c \lambda^{2-p}}{\big| Q_{4\varrho}^{(\lambda)}(z_o) \big|}
	\iint_{Q_{4\varrho}^{(\lambda)}(z_o) \cap E} G \,\dx\dt.
\end{align*}
Therefore, applying Lemma \ref{lem:boundary-Poincare} slice-wise to estimate the first and second term on the right-hand side of the predecing inequality, and using \eqref{eq:boundary-rev-Holder-scaling-assumption}, we deduce that
\begin{align*}
	\sup_{t \in \Lambda_{2\varrho}^{(\lambda)}(t_o)\cap(0,T)} &\bint_{B_{2\varrho}(x_o)} \bigg| \frac{u-g}{\varrho} \bigg|^2(t) \,\dx 
	\\ &\leq
	\frac{c\lambda^{2-p}}{\big| Q_{48\varrho}^{(\lambda)}(z_o) \big|} \iint_{Q_{48\varrho}^{(\lambda)}(z_o) \cap E} |Du|^p + G \,\dx\dt
	\\ &\phantom{=}
	+ \Bigg( \frac{c}{\big| Q_{48\varrho}^{(\lambda)}(z_o) \big|} \iint_{Q_{48\varrho}^{(\lambda)}(z_o) \cap E} |Du|^p + G \,\dx\dt \Bigg)^\frac{2}{p}
	\\ &\phantom{=}
	+ \frac{c \lambda^{2-p}}{\big| Q_{48\varrho}^{(\lambda)}(z_o) \big|}
	\iint_{Q_{48\varrho}^{(\lambda)}(z_o) \cap E} G \,\dx\dt
	\\ &\leq
	c \lambda^2.
\end{align*}
Inserting the preceding estimates into \eqref{eq:rev-Holder-boundary-aux2}, we obtain that
\begin{align}\label{est-I-sigma}
	\mathrm{I}_\sigma
	&\leq
	\frac{c \lambda^{p-q}}{\big| Q_{48\varrho}^{(\lambda)}(z_o) \big|} \iint_{Q_{48\varrho}^{(\lambda)}(z_o) \cap E} |Du|^q + G^\frac{q}{p} \,\dx\dt
\end{align}
in the case $\sigma\ge p$. It remains to consider a
parameter 
$\sigma\in\{2,p\}$ with $\sigma<p$, which only occurs if
$\sigma=2<p$. In this case we use Young's inequality with exponents
$\frac{p}{p-2}$ and $\frac{p}{2}$, Hölder's inequality, and then \eqref{est-I-sigma} with
$\sigma=p$ to estimate
\begin{align*}
  \mathrm{I}_2
  &\le\eps\lambda^p+\hat c(\eps)\mathrm{I}_p\\
  &\le \eps\lambda^p+\hat c(\eps)\frac{c \lambda^{p-q}}{\big| Q_{48\varrho}^{(\lambda)}(z_o) \big|} \iint_{Q_{48\varrho}^{(\lambda)}(z_o) \cap E} |Du|^q + G^\frac{q}{p} \,\dx\dt
\end{align*}
for any $\eps>0$. Therefore, in any case we obtain 
\begin{align*}
  \mathrm{I}_\sigma
  &\leq \eps\lambda^p+\hat c(\eps)\frac{c \lambda^{p-q}}{\big| Q_{48\varrho}^{(\lambda)}(z_o) \big|} \iint_{Q_{48\varrho}^{(\lambda)}(z_o) \cap E} |Du|^q + G^\frac{q}{p} \,\dx\dt\\
        &\leq
	c \varepsilon \lambda^p
	+ \Bigg( \frac{c \hat{c}(\varepsilon)}{\big| Q_{48\varrho}^{(\lambda)}(z_o) \big|} \iint_{Q_{48\varrho}^{(\lambda)}(z_o) \cap E} |Du|^q + G^\frac{q}{p} \,\dx\dt \Bigg)^\frac{p}{q}.
\end{align*}
In the last step we applied Young's inequality with exponents $\frac{p}{q}$ and $\frac{p}{p-q}$.
Using the preceding estimate in \eqref{eq:rev-Holder-boundary-aux1}, we end up with
\begin{align*}
	\frac{1}{\big| Q_\varrho^{(\lambda)}(z_o) \big|} &\iint_{Q_\varrho^{(\lambda)}(z_o) \cap E} |Du|^p \,\dx\dt
	\nonumber \\ &\leq
	c \varepsilon \lambda^p
	+ \Bigg( \frac{c \hat{c}(\varepsilon) }{\big| Q_{48\varrho}^{(\lambda)}(z_o) \big|} \iint_{Q_{48\varrho}^{(\lambda)}(z_o) \cap E} |Du|^q \,\dx\dt \Bigg)^\frac{p}{q}
	\\ &\phantom{=}
	+ \frac{c \hat{c}(\varepsilon) }{\big| Q_{48\varrho}^{(\lambda)}(z_o) \big|} \iint_{Q_{48\varrho}^{(\lambda)}(z_o) \cap E} G \,\dx\dt.
\end{align*}
Thus, by \eqref{eq:boundary-rev-Holder-scaling-assumption}, we infer
\begin{align*}
	&\frac{1}{\big| Q_\varrho^{(\lambda)}(z_o) \big|} \iint_{Q_\varrho^{(\lambda)}(z_o) \cap E} |Du|^p \,\dx\dt
	\nonumber \\ &\leq
	\frac{c \varepsilon}{\big| Q_\varrho^{(\lambda)}(z_o) \big|} \iint_{Q_\varrho^{(\lambda)}(z_o) \cap E} |Du|^p \,\dx\dt
	+ \Bigg( \frac{c \hat{c}(\varepsilon)}{\big| Q_{48\varrho}^{(\lambda)}(z_o) \big|} \iint_{Q_{48\varrho}^{(\lambda)}(z_o) \cap E} |Du|^q \,\dx\dt \Bigg)^\frac{p}{q}
	\\ &\phantom{=}
	+ \frac{c \hat{c}(\varepsilon)}{\big| Q_{48\varrho}^{(\lambda)}(z_o) \big|} \iint_{Q_{48\varrho}^{(\lambda)}(z_o) \cap E} G \,\dx\dt.
\end{align*}
Choosing $\varepsilon>0$ small enough and absorbing the first term on the right-hand side into the left-hand side, we conclude the proof.
\end{proof}

\subsection{Near the initial boundary} \label{sec:rev-Holder-initial}
In the case \eqref{eq:case-initial}, the cylinder $Q_{8\rho}^{(\lambda)}(z_o)$ does not intersect the lateral boundary.
Therefore, this case can overall be treated as the initial boundary case in cylindrical domains, see \cite[Lemma 5.6]{BP}.
We only need two modifications to adapt the proof to our setting.
First of all, since we consider a general right-hand side $F$, we get additional terms in the scaling assumption \eqref{scaling-initial} and the conclusion of the lemma compared to the estimates for $F \equiv 0$ in \cite{BP}.
Note that these terms are included in the definition of $G$ in~\eqref{eq:definition-G}. 
Further, our growth assumption~\eqref{assumption:A}$_2$ is slightly less general than
the corresponding one in \cite[(2.3)]{BP} in the sense that we do not
allow an additive constant on the right-hand side.
Therefore, in contrast to \cite{BP}, we do not have to add the constant $1$ in the definition of $G$.

As in \cite{BP}, our scaling~\eqref{scaling-initial} involves the lateral boundary values via the definition of $G$ in~\eqref{eq:definition-G}.
This has the advantage that we can use the same scaling close to the initial boundary as close to the lateral boundary in Section \ref{sec:reverse-Holder-application}.

\begin{lemma} \label{lem:reverse-holder-initial}
Suppose that $u$ is a weak solution to~\eqref{eq:pde} according to Definition~\ref{def:weak-sol-global}. Moreover, let $Q_{\rho}^{(\lambda)} = Q_{\rho}^{(\lambda)}(z_o) \subset \R^{n+1}$ with $\rho \in (0,1]$ and $\lambda > 0$ such that \eqref{eq:case-initial} holds true.
Suppose that
\begin{equation}\label{scaling-initial}
\frac{1}{|Q_{8\rho}^{(\lambda)}|} \iint_{Q_{8\rho}^{(\lambda)}\cap E} |Du|^p + G \d x \d t \leq \lambda^p \leq \frac{1}{|Q_{\rho}^{(\lambda)}|} \iint_{Q_{\rho}^{(\lambda)}\cap E} |Du|^p + G \, \d x \d t,
\end{equation}
where $G$ is given by \eqref{eq:definition-G}.
Then, there exists a constant $c = c(n,p,C_o,C_1) > 0$ such that 
\begin{align*}
	\frac{1}{\big|Q_{\rho}^{(\lambda)} \big|} &\iint_{Q_{\rho}^{(\lambda)}\cap E} |Du|^p \, \d x \d t
	\\ &\leq
	\left( \frac{c}{\big| Q_{4\rho}^{(\lambda)} \big|} \iint_{Q_{4\rho}^{(\lambda)}\cap E} |D u|^q \, \d x \d t \right)^\frac{p}{q}
	+ \frac{c}{\big| Q_{4\rho}^{(\lambda)} \big|} \iint_{Q_{4\rho}^{(\lambda)} \cap E} G \, \d x \d t
	\\ &\phantom{=}
	+ c \lambda^{p - \hat p} \left( \bint_{B_{4 \rho}} |D g_o|^{\hat{p}_*} \, \d x \right)^\frac{\hat{p}}{\hat{p}_*}.
\end{align*}
holds true for every $\max \{ p-1, \hat p_* \} \leq q < p$.
\end{lemma}

\subsection{Interior case} \label{sec:rev-Holder-interior}
Now, we consider the case \eqref{eq:case-interior}, in which
$Q_{2\rho}^{(\lambda)}(z_o)\subset E$. Since this cylinder neither
intersects the lateral nor the initial boundary, we can use the same
arguments as in the interior case. Therefore, we obtain the following
lemma, whose proof is a straightforward modification of the arguments
leading to \cite[Lemma~13]{Boegelein:1}. Similarly to the
initial boundary case in Section \ref{sec:rev-Holder-initial}, we consider the scaling~\eqref{scaling-interior} involving the
lateral boundary data in order to make it compatible with the
scalings in the other two cases treated in Sections \ref{sec:rev-Holder-lateral} \& \ref{sec:rev-Holder-initial}. 

\begin{lemma} \label{lem:reverse-holder-interior}
Let  $u$ be a weak solution to~\eqref{eq:pde}
such that $Q_{2\rho}^{(\lambda)}(z_o) \subset E$
with $\rho \in (0,1]$ and $\lambda > 0$. Suppose that 
\begin{equation}
	\biint_{ Q_{2\rho}^{(\lambda)} } |D u|^p + G \, \d
x \d t \leq \lambda^p \leq \biint_{Q_\rho^{(\lambda)}}
|Du|^p +  G \, \d x\d t\label{scaling-interior}
\end{equation}
holds true, where $G$ is given by \eqref{eq:definition-G}. Then, there
exists a positive constant $c = c (n,p,C_o,C_1)$ such that
\begin{align*}
\biint_{Q_\rho^{(\lambda)}} |Du|^p\, \d x\d t &\leq c \left( \biint_{Q_{2\rho}^{(\lambda)} } |Du|^q\, \d x\d t \right)^\frac{p}{q} + c \biint_{Q_{2\rho}^{(\lambda)}} G \, \d x \d t,
\end{align*}
for every $\max \{p-1, \hat{p}_*\} \leq q < p$.
\end{lemma}

\section{Proof of Theorem \ref{thm:main}}
\label{sec:proof-of-mainthm}
Fix a point $\tilde{z}_o \in \overline{E}$ and radii $0 < R \leq R_1 < R_2 \leq 2R \leq 2$, and consider concentric cylinders $Q_R(\tilde{z}_o) \subset Q_{R_1}(\tilde{z}_o) \subset Q_{R_2}(\tilde{z}_o) \subset Q_{2R}(\tilde{z}_o)$.
For simplicity, we will omit $\tilde{z}_o$ in the the notation from
now on.
Observe that 	
$$
Q_{2 \rho}(z_o) \subset Q_{2R} (z_o) \subset Q_{4R}
$$
for every $z_o \in Q_{2R}$ and $\rho \leq R$.	
Further, let
$$
\lambda_o = \left( 1+ \frac{1}{|Q_{4R}|}
 \iint_{Q_{4R} \cap E} |Du|^p + G \, \d x \d t \right)^\frac{d}{p},
$$
where $d$ is given by \eqref{eq:scaling-deficit}, and for $\lambda > \lambda_o$ and $r\in (0,2R)$ let
$$
\mathbf{E}(r,\lambda) = \left\{ z \in Q_{r} \cap E : z \text{ is a Lebesgue point of } |Du| \text{ and } |Du|(z) > \lambda \right\}.
$$

We define the radius
$$
R_o = \min \left\{ 1, \lambda^{(p-2) \frac{\beta}{2\beta -1}} \right\} R
$$
and observe that for every $r\in(0,R_o]$ and $z_o\in Q_{2R}$, we have
$Q_r^{(\lambda)}(z_o)\subset Q_{4R}$. 
We consider levels
\begin{equation} \label{eq:lambda-geq-Blambdao}
\lambda \geq B \lambda_o, \quad \text{ where }\quad B= \left( \frac{4 \hat c R}{R_2-R_1} \right)^{(n+2)\frac{d}{p}},
\end{equation}
with 
\begin{equation*}
  \hat{c} := \max \left\{ 240 , 48 \big( 2^{7\beta - 3} M \big)^\frac{1}{2\beta - 1} \right\}.
\end{equation*}
With these choices, we fix a point $z_o \in \mathbf{E}(R_1, \lambda)$.
First, we consider the case of a radius $r$ with $\frac{R_2-R_1}{\mathfrak{m}}
< r \leq R_o$, where we have set 
$$
	\mathfrak{m} =  \frac{\hat c}{\min \left\{ 1, \lambda^{(p-2) \frac{\beta}{2\beta -1}} \right\}}.
$$
Using the fact $B_r^{(\lambda)}(z_o)\subset B_{4R}$ and the definition
of $\lambda_o$, we estimate
\begin{align*}
   \frac{1}{\big|Q_r^{(\lambda)}\big|}
   \iint_{Q_r^{(\lambda)}(z_o)\cap E}  |D u|^p +G  \, \d x\d t
   &\leq
     \frac{|Q_{4R}|}{\big|Q_r^{(\lambda)}\big|}\frac{1}{|Q_{4R}|}
     \iint_{Q_{4R}\cap E} |D u|^p + G   \,\d x\d t \\
&\leq  \left( \frac{4R}{r} \right)^{n+2} \lambda^{p-2} \lambda_o^\frac{p}{d}.
\end{align*}
Now we distinguish between the cases $p<2$ and $p\ge2$. 
In the first case, we use the bound $r>\frac{R_2-R_1}{\mathfrak m}$
and the fact 
$-2-(p-2)(n+2)\frac{\beta}{2 \beta-1}=-\frac{p}{d}<0$, which follows from
assumption~\eqref{eq:beta-range}. This gives
\begin{align*}
  \left( \frac{4R}{r} \right)^{n+2} \lambda^{p-2} \lambda_o^\frac{p}{d}
&<  \left( \frac{4 \hat c R}{R_2-R_1} \right)^{n+2} \lambda^{p}\lambda^{-2-(p-2)(n+2) \frac{\beta}{2 \beta-1}} \lambda_o^\frac{p}{d} \\
&\leq  \left( \frac{4 \hat c R}{R_2-R_1} \right)^{n+2} \lambda^{p} \left[ \left(\frac{4 \hat c  R}{R_2-R_1} \right)^{(n+2)\frac{d}{p}} \lambda_o \right]^{-\frac{p}{d}} \lambda_o^\frac{p}{d} \\
&= \lambda^p.
\end{align*}
Similarly, in the case $p \geq 2$ we obtain
\begin{align*}
  \left( \frac{4R}{r} \right)^{n+2} \lambda^{p-2} \lambda_o^\frac{p}{d} 
&<  \left( \frac{4 \hat c R}{R_2-R_1} \right)^{n+2} \lambda^{p}\lambda^{-2} \lambda_o^\frac{p}{d} \\
&\leq  \left( \frac{4 \hat c R}{R_2-R_1} \right)^{n+2} \lambda^{p} \left[ \left(\frac{4 \hat c  R}{R_2-R_1} \right)^{(n+2)\frac{d}{p}} \lambda_o \right]^{-\frac{p}{d}} \lambda_o^\frac{p}{d} \\
&= \lambda^p,
\end{align*}
since in this case, we have $\frac{p}{d}=2$. Combining the three
preceding estimates, in any case we deduce
\begin{align*}
	\frac{1}{\big|Q_r^{(\lambda)}\big|}\iint_{Q_r^{(\lambda)}(z_o)\cap
  E}  |D u|^p +G   \, \d x\d t
	<\lambda^p,
\end{align*}
provided $\frac{R_2-R_1}{\mathfrak{m}}<r\leq R_o$.

On the other hand, we find that
$$
\liminf_{r \downarrow 0} \frac{1}{\babs{ Q_r^{(\lambda)} }} \iint_{Q_r^{(\lambda)}(z_o) \cap E} |Du|^p + G \, \d x \d t \geq |Du|^p (z_o) > \lambda^p.
$$
Thus, for each $z_o \in \mathbf{E}(R_1,\lambda)$ there exists a maximal radius $\rho_{z_o} \in \left( 0, \frac{R_2-R_1}{\mathfrak{m}} \right]$ such that 
\begin{equation} \label{eq:stopping-rho}
	\frac{1}{\left| Q_{\rho_{z_o}}^{(\lambda)} \right|}
	\iint_{Q_{\rho_{z_o}}^{(\lambda)}(z_o) \cap E} |Du|^p + G \, \d x \d t
	= \lambda^p,
\end{equation}
and 
\begin{equation} \label{eq:stopping-rho-sub}
	\frac{1}{\left| Q_{r}^{(\lambda)} \right|}
	\iint_{Q_{r}^{(\lambda)}(z_o) \cap E} |Du|^p + G \, \d x \d  t
	< \lambda^p
	\quad \text{ for every } r \in (\rho_{z_o}, R_o].
\end{equation}
In this context, observe that $Q_{\hat c \rho_{z_o}}^{(\lambda)}(z_o) \subset Q_{R_2}$ for any $z_o \in Q_{R_1}$.
Indeed, on the one hand, we find that $\hat{c} \rho_{z_o} \leq R_2 - R_1$.
On the other hand, since $\rho_{z_o} \leq \mathfrak{m}^{-1} (R_2-R_1)$, $\lambda \geq 1$, and since we have that $\min\left\{1,\lambda^{(p-2) \frac{\beta}{2\beta -1}}\right\} \leq \lambda^\frac{p-2}{2}$ in the case $p < 2$ by the fact that $\frac{\beta}{2\beta -1} > \frac12$, we obtain that
$$
R_1^2  + \lambda^{2-p} (\hat c \rho_{z_o})^2 \leq R_1^2 + \lambda^{2-p} \min \left\{1,\lambda^{(p-2) \frac{\beta}{2\beta -1}}\right\}^2 (R_2-R_1)^2  \leq R_2^2.
$$

\subsection{Reverse H\"older inequality}
\label{sec:reverse-Holder-application}
Let $\lambda$ satisfy~\eqref{eq:lambda-geq-Blambdao}, fix $z_o \in \mathbf{E} (R_1,\lambda)$ and let $\rho_{z_o} \in \left( 0, \frac{R_2-R_1}{\mathfrak{m}} \right]$ be the maximal radius such that \eqref{eq:stopping-rho} and \eqref{eq:stopping-rho-sub} hold true.
In order to derive a reverse Hölder inequality in $Q_{\rho_{z_o}}^{(\lambda)} = Q_{\rho_{z_o}}^{(\lambda)}(z_o)$, we distinguish between the cases of lateral, initial, and interior cylinders.

First, assume that~\eqref{eq:case-lateral} holds true with $\rho_{z_o}$ in place of $\rho$.
  Note that our assumptions imply $\rho_{z_o} \leq 1$, $\lambda\ge1$ and
$\lambda^{2-p}\rho_{z_o}^2\le1$. Moreover, since
$\hat{c} \geq 48\left( 2^{7\beta -3}
  M \right)^{\frac{1}{2\beta-1}}$, assumption
\eqref{eq:radius-small} is satisfied and we have $48\rho_{z_o}\le
R_o$. In view of~\eqref{eq:stopping-rho}
and~\eqref{eq:stopping-rho-sub}, the latter ensures that the scaling
assumption \eqref{eq:boundary-rev-Holder-scaling-assumption} of
Lemma~\ref{lem:reverse-holder-lateral} is satisfied.
Applying the lemma, we obtain that
\begin{align*}
\frac{1}{\babs{Q_{\rho_{z_o}}^{(\lambda)}}} \iint_{Q_{\rho_{z_o}}^{(\lambda)} \cap E} |Du|^p\, \d x\d t &\leq \left( \frac{c}{\babs{ Q_{48\rho_{z_o}}^{(\lambda)}}}\iint_{Q_{48\rho_{z_o}}^{(\lambda)} \cap E} |Du|^q\, \d x\d t \right)^\frac{p}{q} \\
&\quad+ \frac{c}{\babs{ Q_{48\rho_{z_o}}^{(\lambda)}}} \iint_{Q_{48\rho_{z_o}}^{(\lambda)} \cap E} G \, \d x \d t
\end{align*}
for every $\max\{\gamma, \hat p_*\} \leq q < p$.

Next, if~\eqref{eq:case-initial} holds true with $\rho = \rho_{z_o}$,
the assumptions of Lemma~\ref{lem:reverse-holder-initial} are
satisfied. The lemma gives us that
\begin{align*}
\frac{1}{|Q_{\rho_{z_o}}^{(\lambda)}|} \iint_{Q_{\rho_{z_o}}^{(\lambda)}\cap E} |Du|^p \, \d x \d t &\leq \left( \frac{c}{|Q_{4\rho_{z_o}}^{(\lambda)}|} \iint_{Q_{4\rho_{z_o}}^{(\lambda)}\cap E} |D u|^q \, \d x \d t \right)^\frac{p}{q} \\
&\quad + \frac{c}{|Q_{4\rho_{z_o}}^{(\lambda)}|} \iint_{Q_{4\rho_{z_o}}^{(\lambda)} \cap E} G \, \d x \d t\\
&\quad + c \lambda^{p - \hat p} \left( \bint_{B_{4 \rho_{z_o}}} |D g_o|^{\hat{p}_*} \, \d x \right)^\frac{\hat{p}}{\hat{p}_*},
\end{align*}
for every $\max \{ p-1, \hat p_* \} \leq q < p$.

Finally, in the interior case~\eqref{eq:case-interior} with $\rho =
\rho_{z_o}$, we apply Lemma~\ref{lem:reverse-holder-interior}. We obtain
\begin{align*}
\biint_{Q_{\rho_{z_o}}^{(\lambda)}} |Du|^p\, \d x\d t &\leq c \left( \biint_{Q_{2\rho_{z_o}}^{(\lambda)} } |Du|^q\, \d x\d t \right)^\frac{p}{q} + c \biint_{Q_{2\rho_{z_o}}^{(\lambda)}} G \, \d x \d t,
\end{align*}
for every $\max \{p-1, \hat{p}_*\} \leq q < p$. By taking into account all the cases, we end up with
\begin{align}	\label{eq:reverse-Holder-all-cases} 
	&\frac{1}{\babs{Q_{\rho_{z_o}}^{(\lambda)}}} \iint_{Q_{\rho_{z_o}}^{(\lambda)} \cap E} |Du|^p\, \d x\d t\\
	&\qquad\leq
	\left( \frac{c}{\babs{ Q_{48\rho_{z_o}}^{(\lambda)}}}\iint_{Q_{48\rho_{z_o}}^{(\lambda)} \cap E} |Du|^q\, \d x\d t \right)^\frac{p}{q}
	+ \frac{c}{\babs{ Q_{48\rho_{z_o}}^{(\lambda)}}} \iint_{Q_{48\rho_{z_o}}^{(\lambda)} \cap E} G \, \d x \d t
   \nonumber \\
        &\qquad\qquad
	+  c \lambda^{p - \hat p} \left( \frac{1}{\big|B_{4\rho_{z_o}}\big|}\int_{B_{4
                       \rho_{z_o}}\cap E^0} |D g_o|^{\hat{p}_*}  \, \d x \right)^\frac{\hat{p}}{\hat{p}_*}\chi_{\Lambda_{2\rho_{z_o}}^{(\lambda)}(t_o)}(0),
	\nonumber
\end{align}
for all $\max\{p-1,\gamma, \hat p_*\} \leq q < p$, with a constant $c=c(n,p,C_o,C_1,\alpha)$.

\subsection{Estimates on super-level sets}
For $r \in (0, 2R)$ denote
$$
\mathbf{G}(r,\lambda) = \left\{ z \in Q_{r}\cap E: z\text{ is a Lebesgue point of $G$ and } G(z) > \lambda^p  \right\}
$$
and 
$$
\mathbf{G}_o (r,\lambda) = \left\{ x \in B_{r}\cap E^0: x\text{ is a
    Lebesgue point of $Dg_o$ and } |D g_o|(x) > \lambda \right\}.
$$
Using~\eqref{eq:stopping-rho}, \eqref{eq:reverse-Holder-all-cases}, H\"older's inequality, and~\eqref{eq:stopping-rho-sub}, for $\eta \in (0,1)$, $z_o \in \mathbf{E} (R_1,\lambda)$, and any exponent $q \in [\max\{p-1,\gamma, \hat p_*\}, p)$ we find that
\begin{align*}
	\lambda^p
	&=
	\frac{1}{|Q_{\rho_{z_o}}^{(\lambda)}|}\iint_{Q_{\rho_{z_o}}^{(\lambda)} \cap E} |Du|^p + G  \, \d x \d t \\
	&\leq
	\left( \frac{c}{\babs{ Q_{48\rho_{z_o}}^{(\lambda)}}}\iint_{Q_{48\rho_{z_o}}^{(\lambda)} \cap E} |Du|^q\, \d x\d t \right)^\frac{p}{q}
	+ \frac{c}{\babs{ Q_{48\rho_{z_o}}^{(\lambda)}}} \iint_{Q_{48\rho_{z_o}}^{(\lambda)} \cap E} G \, \d x \d t \\
	&\phantom{=}
	+c  \lambda^{p - \hat p} \left( \frac{1}{|B_{4 \rho_{z_o}}|} \int_{B_{4 \rho_{z_o}}\cap E^0} |D g_o|^{\hat{p}_*} \, \d x \right)^\frac{\hat{p}}{\hat{p}_*} \\
	&\leq
	c \eta^p \lambda^p + \left( \frac{c}{\babs{ Q_{48\rho_{z_o}}^{(\lambda)}}}\iint_{Q_{48\rho_{z_o}}^{(\lambda)} \cap \mathbf{E}(R_2, \eta \lambda)} |Du|^q\, \d x\d t \right)^\frac{p}{q} \\
	&\phantom{=}
	+ \frac{c}{\babs{ Q_{48\rho_{z_o}}^{(\lambda)}}} \iint_{Q_{48\rho_{z_o}}^{(\lambda)} \cap  \mathbf{G}(R_2, \eta \lambda)} G \, \d x \d t \\
	&\phantom{=}
	+ c  \lambda^{p - \hat p} \left( \frac{1}{|B_{4 \rho_{z_o}}|} \int_{B_{4 \rho_{z_o}}\cap \mathbf{G}_o(R_2,\eta \lambda)} |D g_o|^{\hat{p}_*} \, \d x \right)^\frac{\hat{p}}{\hat{p}_*}\\
	&\leq
	c \eta^p \lambda^p +  \frac{c \lambda^{p-q}}{\babs{ Q_{48\rho_{z_o}}^{(\lambda)}}}\iint_{Q_{48\rho_{z_o}}^{(\lambda)} \cap \mathbf{E}(R_2, \eta \lambda)} |Du|^q\, \d x\d t  \\
	&\phantom{=}
	+ \frac{c}{\babs{ Q_{48\rho_{z_o}}^{(\lambda)}}} \iint_{Q_{48\rho_{z_o}}^{(\lambda)} \cap  \mathbf{G}(R_2, \eta \lambda)} G \, \d x \d t \\
	&\phantom{=}
	+ c  \lambda^{p - \hat p} \left( \frac{1}{|B_{4 \rho_{z_o}}|} \int_{B_{4 \rho_{z_o}}\cap \mathbf{G}_o(R_2,\eta \lambda)} |D g_o|^{\hat{p}_*} \, \d x \right)^\frac{\hat{p}}{\hat{p}_*}.
\end{align*}
Choosing $\eta^p = \frac{1}{2c}$ allows us to absorb the first term on the right-hand side into the left.
Using~\eqref{eq:stopping-rho-sub} and the fact that $\hat{c} \rho_{z_o} \leq R_o$, we estimate the left-hand side of the resulting inequality from below by
\begin{align*}
\lambda^p \geq \frac{1}{\big|Q_{\hat c \rho_{z_o}}^{(\lambda)}\big|}\iint_{Q_{\hat c \rho_{z_o}}^{(\lambda)}\cap E} |Du|^p  \, \d x \d t.
\end{align*}
We multiply both sides of the resulting inequality by $\big|Q_{\hat c
  \rho_{z_o}}^{(\lambda)}\big|$. Taking into account that $\lambda^{p-\hat{p}} |Q_{\hat c \rho_{z_o}}^{(\lambda)}| / |B_{4\rho_{z_o}}|^\frac{\hat p}{\hat p_*} \leq c \hat c^{n+2} \rho_{z_o}^{n+2-n \frac{\hat p}{\hat{p}_*}} \lambda^{2-\hat{p}}$ and $\lambda^{2-\hat{p}} \leq 1$, this leads us to
\begin{align}
	\iint_{Q_{\hat c \rho_{z_o}}^{(\lambda)}\cap E} &|Du|^p \, \d x \d t
	\nonumber \\
	&\leq c \hat c^{n+2} \iint_{Q_{48\rho_{z_o}}^{(\lambda)} \cap \mathbf{E}(R_2, \eta \lambda)}  \lambda^{p-q} |Du|^q\, \d x\d t  
	\label{eq:estimate-superlevel-set} \\ &\phantom{=}
	+ c \hat c^{n+2} \iint_{Q_{48\rho_{z_o}}^{(\lambda)} \cap  \mathbf{G}(R_2, \eta \lambda)} G \, \d x \d t
	\nonumber \\ &\phantom{=}
	+ c \hat c^{n+2} \rho_{z_o}^{n+2-n \frac{\hat p}{\hat{p}_*}} \left(  \int_{B_{4 \rho_{z_o}}\cap \mathbf{G}_o(R_2,\eta \lambda)} |D g_o|^{\hat{p}_*} \, \d x \right)^\frac{\hat{p}}{\hat{p}_*}.
	\nonumber
\end{align}
In the following, we cover $\mathbf{E}(R_1,\lambda)$ by cylinders $\big\{ Q_{48\rho_{z_o}}^{(\lambda)}(z_o) \big\}_{z_o \in \mathbf{E}(R_1,\lambda)}$. By the Vitali covering property, see e.g.~\cite[Theorem 1.2]{Heinonen-book}, there exists a countable, pairwise disjoint collection $\big\{Q_{48\rho_{z_i}}^{(\lambda)}(z_i)\big\}_{i \in \N}$ such that
$$
\mathbf{E}(R_1, \lambda) \subset \bigcup_{i \in \N} Q_{\hat c \rho_{z_i}}^{(\lambda)}(z_i) \subset Q_{R_2}.
$$
Here, we have used that $\hat{c} \geq 240$.
Thus, by \eqref{eq:estimate-superlevel-set} we have that
\begin{align*}
\iint_{\mathbf{E}(R_1, \lambda)} |Du|^p \, \d x \d t &\leq c \iint_{\mathbf{E}(R_2, \eta \lambda)}  \lambda^{p-q} |Du|^q\, \d x\d t + c \iint_{ \mathbf{G}(R_2, \eta \lambda)} G \, \d x \d t \\
&\quad + c (R_2-R_1)^{n+2-n \frac{\hat p}{\hat p_*}} \left(  \int_{ \mathbf{G}_o(R_2,\eta \lambda)} |D g_o|^{\hat{p}_*} \, \d x \right)^\frac{\hat{p}}{\hat{p}_*} ,
\end{align*}
with a constant $c=c(n,p,C_o,C_1,\alpha,\hat c)=c(n,p,C_o,C_1,\alpha,\beta,M)$.
Observe that 
\begin{align*}
\iint_{\mathbf{E}(R_1, \eta \lambda) \setminus \mathbf{E}(R_1, \lambda)} |Du|^p  \, \d x \d t \leq \iint_{\mathbf{E}(R_2, \eta \lambda)} \lambda^{p-q} |Du|^q \, \d x \d t.
\end{align*}
Combining this with the penultimate estimate, and replacing
$\eta\lambda$ by $\lambda$ to simplify notation, we deduce
\begin{align*}
\iint_{\mathbf{E}(R_1, \lambda)} |Du|^p \, \d x \d t &\leq c \iint_{\mathbf{E}(R_2,  \lambda)}  \lambda^{p-q} |Du|^q\, \d x\d t + c \iint_{ \mathbf{G}(R_2,  \lambda)} G \, \d x \d t \\
&\quad + c (R_2-R_1)^{n+2-n \frac{\hat p}{\hat p_*}} \left(  \int_{ \mathbf{G}_o(R_2,\lambda)} |D g_o|^{\hat{p}_*} \, \d x \right)^\frac{\hat{p}}{\hat{p}_*} \\
&\leq c \iint_{\mathbf{E}(R_2,  \lambda)}  \lambda^{p-q} |Du|^q\, \d x\d t + c \iint_{ \mathbf{G}(R_2,  \lambda)} G \, \d x \d t \\
&\quad + c |Q_{R_2}| \left( \frac{1}{|B_{R_2}|} \int_{ \mathbf{G}_o(R_2,\lambda)} |D g_o|^{\hat{p}_*} \, \d x \right)^\frac{\hat{p}}{\hat{p}_*}
\end{align*}
for any $\lambda \geq \eta B \lambda_o$, where $c=c(n,p,C_o,C_1,\alpha,\beta,M)$.
We denote
$$
|Du|_k = \min \{ |Du| , k \},
$$
and
$$
\mathbf{E}_k (r,\lambda) = \left\{ z \in Q_r \cap E: |Du|_k(z) > \lambda \right\}.
$$
Note that $\mathbf{E}_k (r,\lambda) = \varnothing$ if $k \leq \lambda$, and $\mathbf{E}_k (r,\lambda) = \mathbf{E} (r,\lambda)$ if $k > \lambda$.
Therefore, in the truncated level sets, the last inequality above implies
\begin{align*}
\iint_{\mathbf{E}_k(R_1, \lambda)} |Du|_k^{p-q} |Du|^q \, \d x \d t
&\leq c \iint_{\mathbf{E}_k(R_2,  \lambda)}  \lambda^{p-q} |Du|^q\, \d x\d t + c \iint_{ \mathbf{G}(R_2,  \lambda)} G \, \d x \d t \\
&\quad + c |Q_{R_2}| \left( \frac{1}{|B_{R_2}|} \int_{ \mathbf{G}_o(R_2,\lambda)} |D g_o|^{\hat{p}_*} \, \d x \right)^\frac{\hat{p}}{\hat{p}_*}.
\end{align*}
We multiply the preceding inequality by $\lambda^{\eps - 1}$ and integrate over $(\lambda_1,\infty)$, where $\lambda_1 := \eta B \lambda_o$.
Since $\mathbf{G}_o (R_2,\lambda) \subset \mathbf{G}_o (R_2,\lambda_1)$ for $\lambda > \lambda_1$, and by Fubini's theorem and H\"older's inequality, we obtain that
\begin{align*}
	&\int_{\lambda_1}^\infty \lambda^{\eps-1} \left( \frac{1}{|B_{R_2}|} \int_{ \mathbf{G}_o (R_2,\lambda)} |D g_o|^{\hat{p}_*} \, \d x \right)^\frac{\hat{p}}{\hat{p}_*} \, \d \lambda \\
	&\quad \leq
	\left( \frac{1}{|B_{R_2}|} \int_{ \mathbf{G}_o (R_2,\lambda_1)} |D g_o|^{\hat{p}_*} \, \d x \right)^{\frac{\hat{p}}{\hat{p}_*}-1}
	\int_{\lambda_1}^\infty \lambda^{\eps-1} \frac{1}{|B_{R_2}|} \int_{ \mathbf{G}_o (R_2,\lambda)} |D g_o|^{\hat{p}_*} \, \d x \, \d \lambda \\
&\quad \leq \frac{1}{\eps}\left( \frac{1}{|B_{R_2}|} \int_{ \mathbf{G}_o (R_2,\lambda_1)} |D g_o|^{\hat{p}_*+\eps} \, \d x \right)^{\frac{\hat{p} - \hat p_*}{\hat{p}_*+\eps}} \frac{1}{ |B_{R_2}|} \int_{ \mathbf{G}_o (R_2,\lambda_1)} |D g_o|^{\hat p_* + \eps} \, \d x \\
&\quad \leq \frac{1}{\eps}\left( \frac{1}{|B_{R_2}|} \int_{ \mathbf{G}_o (R_2,\lambda_1)} |D g_o|^{\hat{p}_*+\eps} \, \d x \right)^{\frac{\hat{p} +\eps }{\hat{p}_*+\eps}}.
\end{align*}
Analogously, by Fubini's theorem we find that
\begin{align*}
&\int_{\lambda_1}^\infty \lambda^{\eps -1} \iint_{\mathbf{E}_k(R_1,\lambda)} |Du|_k^{p-q} |Du|^q \, \d x\d t \, \d \lambda \\
&\quad = \frac{1}{\eps} \iint_{\mathbf{E}_k(R_1,\lambda_1)} \left( |Du|_k^{p+\eps-q} |Du|^q - \lambda_1^\eps |Du|_k^{p-q} |Du|^q \right) \, \d x \d t,
\end{align*}
that
\begin{align*}
&\int_{\lambda_1}^\infty \lambda^{p-q+\eps-1} \iint_{\mathbf{E}_k(R_2,\lambda)} |Du|^q \, \d x \d t \, \d \lambda \\
&\quad \leq \frac{1}{p-q+\eps} \iint_{\mathbf{E}_k(R_2,\lambda_1)} |Du|_k^{p-q+\eps} |Du|^q \, \d x  \d t,
\end{align*}
and that
\begin{align*}
&\int_{\lambda_1}^\infty \lambda^{\eps-1} \iint_{\mathbf{G}(R_2,\lambda)} G \, \d x \d t \, \d \lambda  \leq \frac{1}{\eps} \iint_{\mathbf{G}(R_2,\lambda_1)} G^{1+ \frac{\eps}{p}} \, \d x  \d t.
\end{align*}
Finally, by definition of $\mathbf{E}_k(R_1,\lambda_1)$ we conclude that
\begin{align*}
	\iint_{(Q_{R_1}\cap E) \setminus \mathbf{E}_k(R_1,\lambda_1)}
  	&|Du|_k^{p+\eps-q} |Du|^q \, \d x\d t
  	\\ & \leq \lambda_1^\eps \int_{(Q_{R_1}\cap E) \setminus \mathbf{E}_k(R_1,\lambda_1)} |Du|_k^{p-q} |Du|^q \, \d x\d t.
\end{align*}
By combining all the estimates we have that
\begin{align*}
  \iint_{Q_{R_1}\cap E} |Du|_k^{p-q+\eps} |Du|^q  \, \d x\d t
  &\leq \frac{c_* \eps}{p-q} \iint_{Q_{R_2}\cap E} |Du|_k^{p-q+\eps} |Du|^q  \, \d x\d t \\
&\quad + \lambda_1^\eps \iint_{Q_{2R}\cap E} |Du|^p \, \d x \d t \\
&\quad + c \iint_{Q_{2R}\cap E} G^{1+\frac{\eps}{p}} \, \d x \d t \\
&\quad + c |Q_{2R}| \left( \frac{1}{|B_{2R}|} \int_{ B_{2R} \cap E^0 } |D g_o|^{\hat{p}_*+\eps} \, \d x \right)^{\frac{\hat{p} +\eps }{\hat{p}_*+\eps}},
\end{align*}
with a constant $c_\ast = c_\ast(n,p,C_o,C_1, M, \alpha,\beta) \geq 1$.
Then, we choose 
$$
\eps_o = \frac{p-q}{2c_*} < 1
$$
and consider $\eps \leq \eps_o$.
Since $\lambda_1^\eps \leq B \lambda_o^\epsilon$, recalling the definition of $B$ we obtain that
\begin{align*}
\iint_{Q_{R_1}\cap E} |Du|_k^{p-q+\eps} |Du|^q  \, \d x\d t &\leq \frac12 \iint_{Q_{R_2}\cap E} |Du|_k^{p-q+\eps} |Du|^q  \, \d x\d t \\
&\quad + c \left( \frac{R}{R_2-R_1} \right)^{(n+2)\frac{d}{p}} \lambda_o^\eps \iint_{Q_{2R}\cap E} |Du|^p  \, \d x \d t \\
&\quad + c \iint_{Q_{2R}\cap E} G^{1+\frac{\eps}{p}} \, \d x \d t \\
&\quad + c |Q_{2R}| \left( \frac{1}{|B_{2R}|} \int_{ B_{2R} \cap E^0 } |D g_o|^{\hat{p}_*+\eps} \, \d x \right)^{\frac{\hat{p} +\eps }{\hat{p}_*+\eps}}.
\end{align*}
By using the iteration lemma, i.e., Lemma~\ref{lem:iteration}, and
passing to the limit $k\to\infty$ by means of Fatou's lemma, we obtain 
\begin{align*}
\iint_{Q_{R}\cap E} |Du|^{p+\eps}\, \d x \d t &\leq c \lambda_o^\eps \iint_{Q_{2R}\cap E} |Du|^p \, \d x \d t + c \iint_{Q_{2R}\cap E} G^{1+\frac{\eps}{p}}  \, \d x \d t \\
&\quad + c|Q_{2R}|  \left( \frac{1}{|B_{2R}|} \int_{ B_{2R} \cap E^0 } |D g_o|^{\hat{p}_*+\eps} \, \d x \right)^{\frac{\hat{p} +\eps }{\hat{p}_*+\eps}}.
\end{align*}
Recalling the definition of $\lambda_o$ and dividing by $|Q_R|$, this concludes the proof of Theorem \ref{thm:main}.

\bigskip
\textbf{Acknowledgements.}
This research was funded in whole or in part by the Austrian Science Fund (FWF) projects \emph{evolutionary problems in noncylindrical domains}, grant DOI 10.55776/J4853, and \emph{widely degenerate partial differential equations}, grant DOI 10.55776/P36295.
The first two authors thank the Faculty of Mathematics of the University of Duisburg-Essen for the hospitality during their research visit.
For the purpose of open access, the authors have applied a CC BY public copyright license to any Author Accepted Manuscript (AAM) version arising from this submission.

\end{document}